\documentclass[a4paper,oneside,10pt,oldfontcommands,reqno]{amsart}
\usepackage[a4paper, total={426pt, 674pt}]{geometry}

\newcommand{\langue}{anglais}	

\usepackage{ifthen}
\ifthenelse{\equal{\langue}{francais}}{
	\usepackage[french]{babel}}
	{\ifthenelse{\equal{\langue}{anglais}}{\usepackage[english]{babel}}{}}
\usepackage[T1]{fontenc}
\usepackage[utf8]{inputenc}
\usepackage[babel,german=guillemets]{csquotes}
\usepackage{eso-pic}

\usepackage[backend=biber, style=alphabetic, sorting=nyt, giveninits=true, doi=false, isbn=false, url=false, eprint=false, maxbibnames=99, maxcitenames=5]{biblatex} 
\DeclareNameAlias{author}{last-first}
\AtEveryBibitem{\clearfield{month}}
\AtEveryBibitem{\clearfield{day}}

\usepackage{lmodern} 
\usepackage{enumitem}
\usepackage{ifthen}
\ifthenelse{\equal{\langue}{francais}}{
	\frenchbsetup{StandardLists=true}}
	{}

\usepackage{graphicx} 
\usepackage{subcaption}

\usepackage{amsmath}
\usepackage{amssymb}
\usepackage{amsthm}
\usepackage{amsfonts}
\usepackage{bbold}
\usepackage{mathtools}
\usepackage{tikz-cd}
\usepackage{hyperref}
\usepackage{multirow}
\usepackage[normalem]{ulem}
\usepackage{cases}
\usepackage{adjustbox}
\usepackage{resizegather}

\ifthenelse{\equal{\langue}{francais}}{
	\newcommand{\theoremenom}{Théorème}
	\newcommand{\propositionnom}{Proposition}
	\newcommand{\lemmenom}{Lemme}
	\newcommand{\corollairenom}{Corollaire}
	\newcommand{\definitionnom}{Définition}
	\newcommand{\remarquenom}{Remarque}
	\newcommand{\exemplenom}{Exemple}
	\newcommand{\conjecturenom}{Conjecture}
}{\ifthenelse{\equal{\langue}{anglais}}{
	\newcommand{\theoremenom}{Theorem}
	\newcommand{\propositionnom}{Proposition}
	\newcommand{\lemmenom}{Lemma}
	\newcommand{\corollairenom}{Corollary}
	\newcommand{\definitionnom}{Definition}
	\newcommand{\remarquenom}{Remark}
	\newcommand{\exemplenom}{Example}
	\newcommand{\conjecturenom}{Conjecture}
}{}}
	
\newtheorem{theoreme}{\theoremenom}[section]
\newtheorem{proposition}[theoreme]{\propositionnom}
\newtheorem{lemme}[theoreme]{\lemmenom}
\newtheorem{corollaire}[theoreme]{\corollairenom}
\newtheorem{definition}[theoreme]{\definitionnom}
\newtheorem{remarque}[theoreme]{\remarquenom}

\newtheorem*{conjecture}{\conjecturenom}

\usepackage{thmtools}
\usepackage{thm-restate}

\makeatletter
\def\cleartheorem#1{%
    \expandafter\let\csname#1\endcsname\relax
    \expandafter\let\csname c@#1\endcsname\relax
}
\makeatother
\newcommand{\compteurThm}{1}
\newcounter{annexe}

\newcommand{\R}{\mathbb{R}}
\newcommand{\N}{\mathbb{N}}
\newcommand{\Z}{\mathbb{Z}}
\newcommand{\dd}{\mathcal{A}}

\newcommand{\ds}{\displaystyle}
\newcommand{\dsum}{\ds\sum}
\newcommand{\eqskip}{ \vspace*{2mm}\\ }

\renewcommand{\(}{\left(}

\newcommand{\ccr}{\color{red}}

\newcommand{\so}{{\rm o}}

\newcommand{\fdsh}{\textup{FDSH}}

\DeclareUnicodeCharacter{0301}{\'{e}}
\DeclareMathOperator{\dist}{dist}
\DeclareMathOperator{\dv}{div}

\begin{document}

\pagestyle{empty} 


\title{Neumann's nodal line may be closed on doubly-connected planar domains}

\author[P. Freitas]{Pedro Freitas}
\author[R. Leylekian]{Roméo Leylekian}
\address{Grupo de F\'{\i}sica Matem\'{a}tica, Instituto Superior T\'{e}cnico, Universidade de Lisboa, Av. Rovisco Pais, 1049-001 Lisboa, Portugal}
\email{pedrodefreitas@tecnico.ulisboa.pt, romeo.leylekian@tecnico.ulisboa.pt}



\begin{abstract}
We show the existence of planar domains with one hole for which the first non-trivial Neumann eigenfunction has a closed nodal line fully contained
inside the domain. This is optimal, as it is known since Pleijel's 1956 result that the nodal line cannot be closed on simply-connected planar
domains. A part of the proof is based on the study of convergence of eigenvalues and eigenfunctions of graph-like domains towards metric graphs.
We improve the known results of convergence of eigenfunctions, by showing a strong transversal convergence.
\end{abstract}
\maketitle



\pagestyle{plain} 

\section{Introduction}

Consider the Neumann eigenvalue problem
\[
\left\{
\begin{array}{lll}
\Delta u + \mu u  =  0 & \text{in }V,\eqskip
 \partial_n u = 0 & \text{on }\partial V,
 \end{array}
\right.
\]
where $V$ is a bounded connected regular (say Lipschitz) planar domain and $\partial_n$ denotes the normal derivative. Under these conditions, it follows that the spectrum consists of a sequence of eigenvalues
\[
 0 = \mu_{1} < \mu_{2}\leq \mu_{3} \dots \to +\infty
\]
with the corresponding eigenfunctions forming a complete orthonormal system of $L^2(V)$. By connectedness, Courant's nodal domain theorem applies and stipulates that any $n$-th eigenfunction has at most $n$ nodal domains. We recall that, for a given function, the nodal domains are the connected components of the set where the function does not vanish, while the nodal line is the closure of its zero set. In particular, Courant's result and the orthogonality of the eigenfunctions imply that the nodal line of a second eigenfunction splits $V$ in exactly two nodal domains.

In a 1956 paper on Courant's theorem~\cite{pleijel}, \AA ke Pleijel stated that the nodal line of any second Neumann eigenfunction had to ``consist of a transverse cutting $V$ into two simply connected
regions.'' His argument also used the fact, which had been proven by P\'{o}lya four years earlier, that the
first non-trivial Neumann eigenvalue is strictly smaller than the first Dirichlet eigenvalue of the same domain. Pleijel then argued that
if the Neumann eigenfunction encircled a subdomain $N_{2}$ of $V$ without touching the boundary $\partial V$, the Neumann and Dirichlet
eigenvalues of $V$ and $N_{2}$, respectively, would be equal. Using the fact that Dirichlet eigenvalues are monotone by inclusion, this would
then lead to a contradiction.

Although not mentioned explicitly in either the argument or the conditions on $V$, it is assumed implicitly in Pleijel's reasoning that
the original domain $V$ is itself simply-connected. Otherwise, if the nodal domain $N_{2}$ encircles one of the {\it holes} in $V$, then
it is no longer possible to reach a contradiction from P\'{o}lya's comparison result due to the presence of Neumann boundary conditions on parts of
$\partial N_{2}$. More precisely, the comparison result between the first Dirichlet and the first non-trivial Neumann eigenvalues used
in the last step in Pleijel's argument now becomes a comparison between the first Neumann-Dirichlet and the first non-trivial
Neumann eigenvalues, and it does no longer necessarilly hold that the former is larger than or equal to the latter. We note, however,
that this argument ensures that any closed nodal line of a second Neumann eigenfunction must encircle a portion of the boundary of the domain.

Probably as a result of this absence of explicitly stated conditions on $V$ for the argument in~\cite{pleijel} to hold, the question of whether
the Neumann nodal line can be closed has either been ignored in the literature or it has even been stated that in the Neumann case the nodal
line associated with a second eigenfunction cannot be closed. An example of the latter may already be found in~\cite[p. 466]{Payne}, where Payne states that in~\cite{pleijel} Pleijel used the inequality between Dirichlet and Neumann eigenvalues mentioned above, together with the ``monotone behaviour of $\lambda_{1}$'' to point out that the first non-trivial Neumann eigenfunction
``can have no closed nodal line.'' It is also given as an exercise in~\cite[Problem 3, p. 128]{bandle1980}. Note further that this imprecision
has percolated in other contexts, e.g. the Steklov problem~\cite{kuttler-sigillito}, for which the issue was corrected only recently~\cite{delatorre-mancini-pistoia-provenzano}.
Thus most of the work on the Neumann
nodal line has focused on its localisation~\cite{atar-burdzy,jerison}, with exceptions being~\cite{Freitas02}, where it was shown that closed
nodal lines in both the Dirichlet and Neumann case do exist on $2-$manifolds, even in the simply-connected case, and~\cite{kennedy18} where
a problem related to certain symmetries of Neumann eigenfunctions was considered.

All of this is in sharp contrast with what happened for the Dirichlet problem, for which not only is the result far from being straightforward
even in the simply-connected case where it remains open but, following a conjecture
formulated by Payne in 1967~\cite[Conjecture 5]{Payne} stating that the second Dirichlet eigenfunction ``cannot have a closed nodal line for
any domain'', there appeared a string of partial results~\cite{Payne2,lin,putter,Melas,d00,
fk08,kiwan,mukherjee-saha}
and counterexamples~\cite{H2ON,fournais01,freitas-krejcirik07,kennedy,dgh21}.
Recently, we have shown that in the Dirichlet case, the number of holes necessary to have a closed nodal line may be brought down to one,
that is, there are doubly-connected domains such that the nodal line encircles the hole while not touching the boundary of the 
domain~\cite{freitas-leylekian}. The intuitive idea behind this construction is based on the fact that having a hole implies the boundary is formed by two disjoint components, say $\Gamma_{1}$ and $\Gamma_{2}$. Because of this, it is then
possible to find a one-parameter family of such domains for which there is a transition between nodal lines that touch only $\Gamma_{1}$ to those 
that touch only $\Gamma_{2}$. By continuity, and provided certain conditions are satisfied (such as simplicity of the second eigenvalue and symmetry of
the eigenfunction), it is to be expected that there will exist a value of the parameter in this transition for which the nodal line 
does not touch either $\Gamma_{1}$ or $\Gamma_{2}$ -- see below for a more complete description, and~\cite{freitas-leylekian} for the result and 
details. Although the Neumann problem turns out to be more complicated, it is possible to start from the same idea to show that there will also 
exist domains in this case for which the nodal line can be closed.

\begin{theoreme}\label{thm:resultat principal}
There exists a doubly-connected bounded open subset of\/ $\R^2$ with a second Neumann eigenfunction whose nodal line is closed and does not touch the boundary.
\end{theoreme}
\begin{figure}[h!]
\centering
\includegraphics[scale=.8]{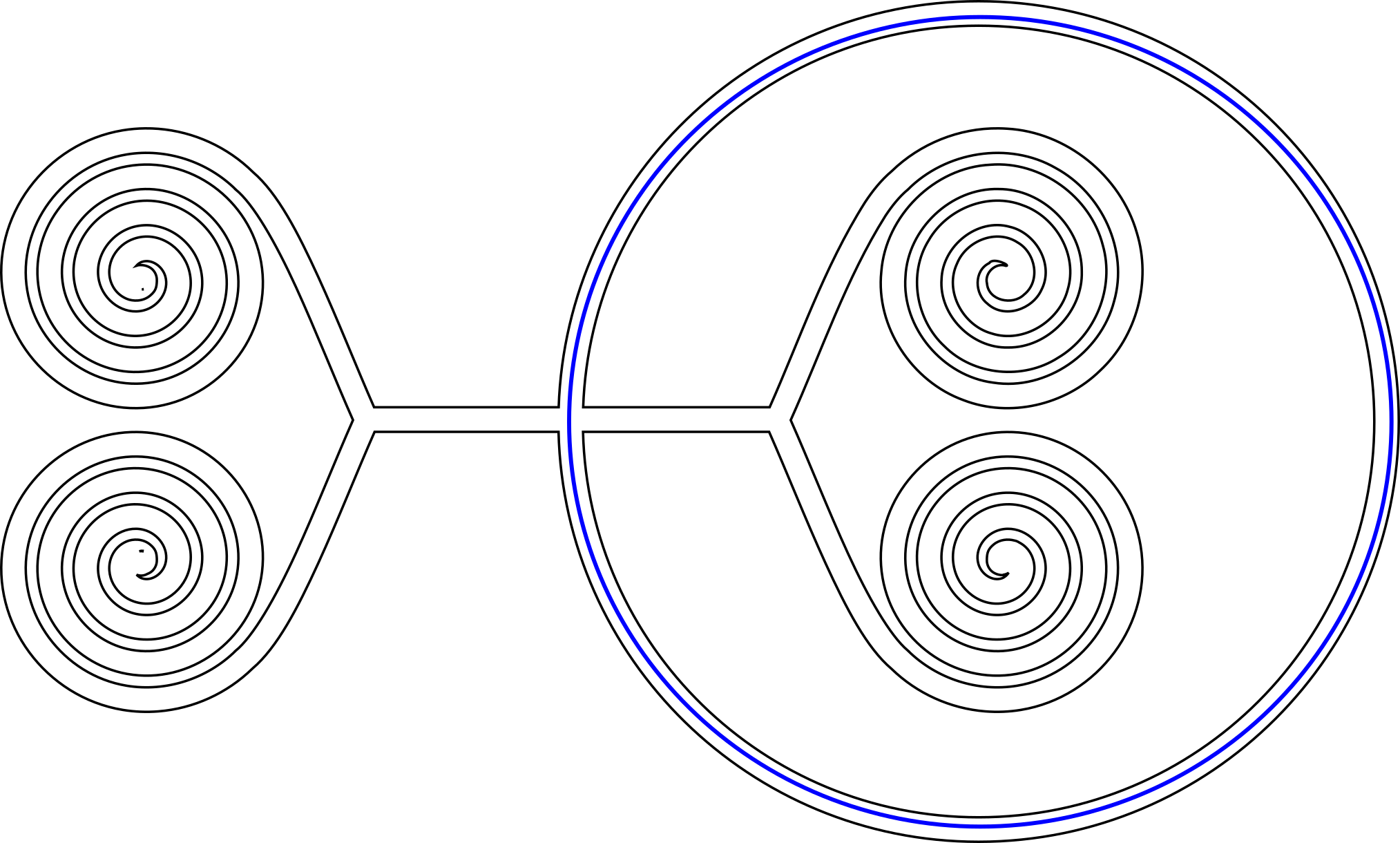}
\caption{Schematic representation of a domain with a second Neumann eigenfunction whose nodal line (in blue) is expected to remain closed.}
\label{fig:contrexemple}
\end{figure}
A schematic representation of a domain for which this phenomenon occurs is shown in Figure~\ref{fig:contrexemple}. Note that by stability of Neumann eigenfunctions
with respect to the excision of small balls~\cite{felli-liverani-ognibene} we also conclude that there exist planar domains with an arbitrary number of holes for which
the Neumann nodal line is closed.

\begin{corollaire}\label{corollaire:plus de trous}
For any integer $N\geq 2$, there exists a bounded connected open subset of\/ $\R^2$ whose boundary is made of $N$ connected components, and with a second Neumann eigenfunction whose nodal line is closed and does not touch the boundary.
\end{corollaire}

\begin{remarque}
The domains in Theorem~\ref{thm:resultat principal} and in Corollary~\ref{corollaire:plus de trous} may further be assumed smooth, symmetric with respect to one direction, and with a simple second eigenvalue. Note also that by small deformations one may generate a continuum of domains, including non-symmetric ones, whose second eigenfunction still has a closed nodal line.
\end{remarque}

Before going further, we would like to point out that our result supports the idea of some parallel, in the behaviour of the nodal line, between a variety of eigenvalue problems associated with the Laplacian. More precisely, we believe that the fact that the nodal line cannot be closed on simply connected planar domains, while it may be on doubly- and multiply-connected domains is a rather general feature. For instance, this dichotomy
was recently established for the Steklov problem by DelaTorre, Mancini, Pistoia and Provenzano in~\cite{delatorre-mancini-pistoia-provenzano}, whose purpose was to reintroduce the assumption of simple connectedness omitted in~\cite{kuttler-sigillito}. In the same way, the present article, combined with Pleijel's argument, shows that the claim holds for the Neumann problem. Besides, our previous paper~\cite{freitas-leylekian} proves the second half of the claim in the Dirichlet case, while the other half is conjectured in~\cite{H2ON}. In the Robin case, the nodal line may be closed on multiply-connected domains according to~\cite{kennedy11}, even if the minimal number of holes needed to construct such a domain remains unknown. Lastly, we would like to mention the paper~\cite{anoop-bobkov-drabek}, where it is shown that, over annuli, the mixed problem obtained by imposing Robin boundary condition on the inner boundary and Neumann condition on the outer one may lead to a circular nodal line, trapped inside the domain. These results suggest (and lend credence) to the following conjecture.

\begin{conjecture}[Topological Nodal Line Conjecture]
Consider the Laplacian on a bounded regular planar domain $\Omega$ with Dirichlet or Robin boundary conditions. Then the nodal line of any second eigenfunction cannot be closed if\/ $\Omega$ is simply-connected, while there exist doubly-connected
domains for which the nodal line is closed and does not touch the boundary.
\end{conjecture}

The present paper consists mostly of the proof of Theorem~\ref{thm:resultat principal}, which is inspired by the \emph{sliding procedure} developed in~\cite{freitas-leylekian}.
However, the Dirichlet problem which is the focus of~\cite{freitas-leylekian} turns out to be considerably simpler than the Neumann problem. The idea of the
sliding procedure in the Dirichlet case is to start from a domain $R$ such as a rectangle for which the nodal line
bisects $R$ vertically. If we then let a wide thin annulus move continuously from the left to the right and consider the domain obtained as 
the union of $R$ with the moving annulus (see Figure~\ref{fig:rectangle}), the eigenfunction will be close to the eigenfunction of $R$. Hence the
nodal line will be close to the vertical segment bisecting $R$. As the annulus moves from the left to the right, the nodal line will touch
different boundary components of the domain. And as the movement is continuous, the nodal line must detach from the boundary at some point, hence it is closed right after that instant. See~\cite{freitas-leylekian} for more details on this sliding procedure.

\begin{figure}[h]
    \centering
    \begin{subfigure}[b]{0.3\textwidth}
        \centering
        \includegraphics[height=.2\textheight]{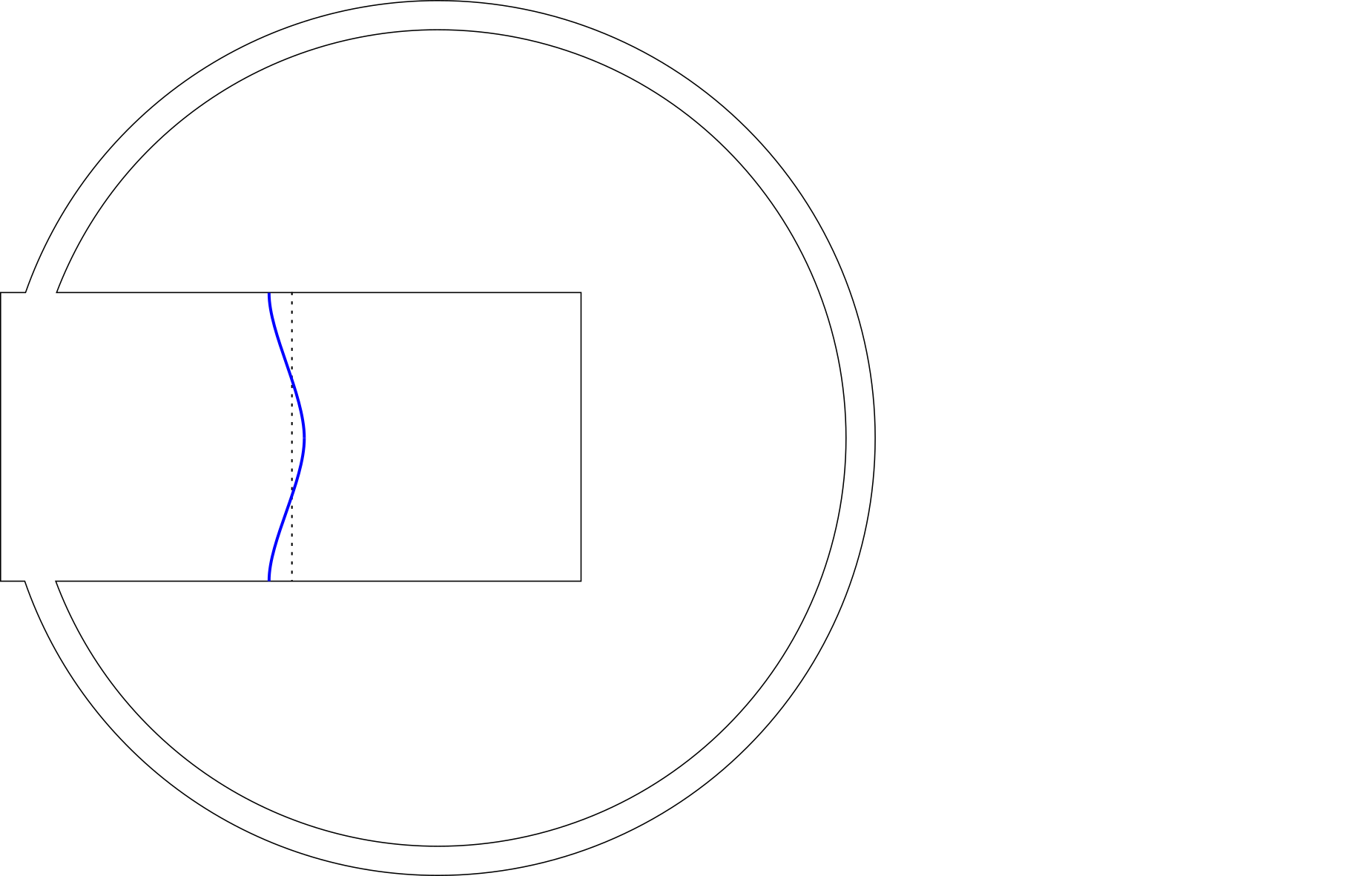}
        
        \label{fig:rectangle t=t--}
    \end{subfigure}
    \hspace{1em}
    \begin{subfigure}[b]{0.3\textwidth}   
        \centering 
        \includegraphics[height=.2\textheight]{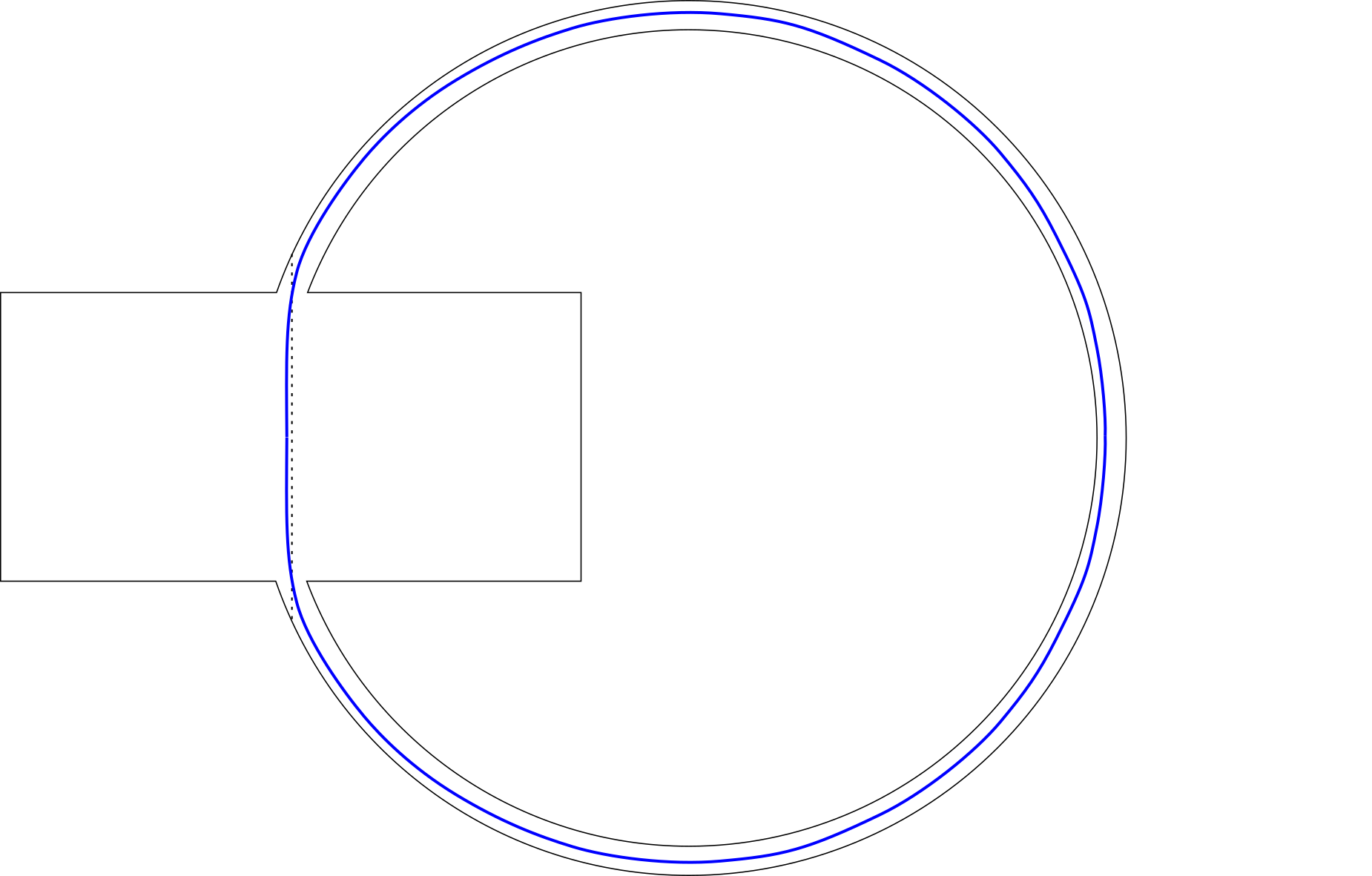}
        
        \label{fig:rectangle t=t0}
    \end{subfigure}
    
    \vskip\baselineskip
    \begin{subfigure}[b]{0.3\textwidth}  
        \centering 
        \includegraphics[height=.2\textheight]{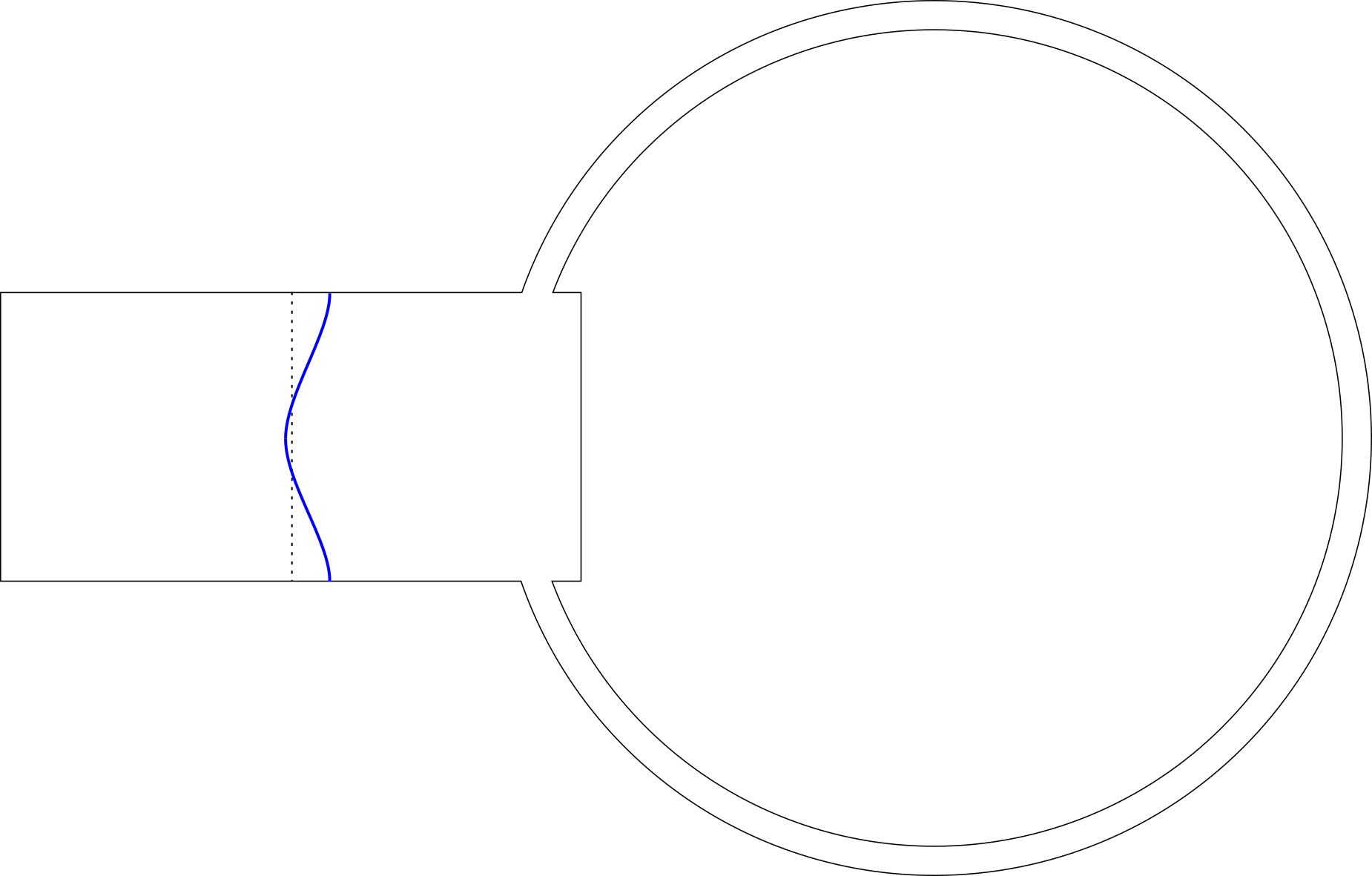}
        
        \label{fig:rectangle t=t++}
    \end{subfigure}
\caption{Illustration of the Dirichlet \fdsh\/ constructed in \cite{freitas-leylekian}. In each case, the
    dotted line is the segment bisecting the rectangle and the blue line is the expected nodal line.}
    \label{fig:rectangle}
\end{figure}

The strategy described above is quite general in the sense that it may be applied to any family of domains whose second eigenfunction's nodal line moves with some regularity from one boundary component to the other. In \cite{freitas-leylekian} this observation motivated the introduction of the concept of \emph{family of domains with a sliding handle} (\fdsh\/) which basically gives a collection of requirements ensuring that the nodal line does behave like this. It then remained to prove that the union of the rectangle and the moving annulus of Figure~\ref{fig:rectangle} fits 
into this framework, namely that it constitutes an~\fdsh. This is what we proved in~\cite{freitas-leylekian}, after smoothing the corners for technical reasons.

For the Neumann problem, in principle the sliding procedure is still valid: if the nodal line touches different boundary components of the domain in the course of the sliding, it means that it should detach at some point. In other words the notion of \fdsh\/ still guarantees the existence of a domain with a closed nodal line. The only difference being that Dirichlet eigenvalues and eigenfunctions must be replaced by Neumann eigenvalues and eigenfunctions in the definition. We will refer to Dirichlet \fdsh\/ and Neumann \fdsh\/ to distinguish between the two settings. The definition of Neumann \fdsh\/ is given in Definition~\ref{def:sliding domains} but we stress that it is exactly the transcription of \cite[Definition~2.1]{freitas-leylekian}, where the notion of Dirichlet \fdsh\/ was formalised.

In the Neumann case however, it is much harder to actually construct an \fdsh. For instance, the Dirichlet \fdsh\/ of Figure~\ref{fig:rectangle} is not a good candidate since it is
expected that in the Neumann problem the nodal line will have two components cutting the annulus transversally, and hence will touch both boundary components, which is excluded in an \fdsh~(see \cite[Lemma~3.1]{freitas-leylekian}). To explain this, let us state the following heuristic
principles for low~order eigenfunctions:
\begin{equation}\label{principle:dirichlet}
\tag{D}
\textit{Dirichlet eigenfunctions concentrate where the domain is not thin.}
\end{equation}
\begin{equation}\label{principle:neumann}
\tag{N}
\textit{Neumann eigenfunctions oscillate along elongated parts of the domain.}
\end{equation}

The first principle confirms a posteriori that it is natural in Figure~\ref{fig:rectangle} to expect the second Dirichlet eigenfunction and its nodal line to concentrate in the rectangle rather than in the annulus. On the contrary, since the latter is long compared to the former, principle~\eqref{principle:neumann} supports the idea according to which the second Neumann eigenfunction is likely to concentrate and to change sign along the annulus.

In view of the above considerations, we understand that for the second Neumann eigenfunction to concentrate in the \enquote{static} part of the domain, it is natural to make that part sufficiently long, compared to the moving annulus. Indeed, if one elongates enough the rectangle horizontally the eigenfunction is likely to concentrate there, and the nodal line is likely to bisect the rectangle vertically, like in the Dirichlet case. In such a configuration the width of the rectangle would probably not play a significant role on the location of the nodal line, hence we may assume that it is of the same order of magnitude as the width of the annulus. This choice simplifies the analysis, especially in the regime of small width. In that regime the domain looks like a graph, similar to that of Figure~\ref{fig:graph simple alonge} -- here and in
what follows, graph vertices are marked as black dots.

This approach to domains that may form Neumann \fdsh\/ lead us naturally to the study of metric graphs and graph-like domains. In the 
present article, metric graphs will always be assumed finite and compact. We refer to \cite{berkolaiko-kuchment} for an introduction to metric 
graphs. In particular it is known \cite[Section~7.5]{berkolaiko-kuchment} that the Neumann Laplacian on graph-like domains converges in some sense 
to the Neumann-Kirchhoff Laplacian on the limit graph. In appendix~\ref{annexe:convergence}, we collect some results in this direction from a paper 
by Post~\cite{post}, and we also prove what we believe to be a new convergence result which we need for our purposes.

The appeal of metric graphs in this context is that it is rather straightforward to derive the equations satisfied by
the spectrum of the Neumann-Kirchhoff Laplacian, consisting of the zeros of the determinant of a matrix with trigonometric entries. The 
corresponding eigenvalues and eigenfunctions may then be approximated numerically if need be. See for instance the discussion in section~\ref{sec:graph with a loop}, where the expression of the eigenfunctions is given explicitly. Numerical calculations confirmed that when the horizontal line is long compared to the circle in Figure~\ref{fig:graph simple alonge}, the nodal set consists of a point located close to the midpoint of that line. Note that this is coherent with principle~\eqref{principle:neumann}. The sliding procedure would then consist in letting slide the circle around the midpoint. However, this is not possible for a planar graph without generating another intersection between the line and the circle.

\begin{figure}[h]
    \centering
    \begin{subfigure}[b]{0.3\textwidth}
        
        \adjustbox{margin=0pt 0pt 0pt 0pt, center}{\includegraphics[height=.09\textheight]{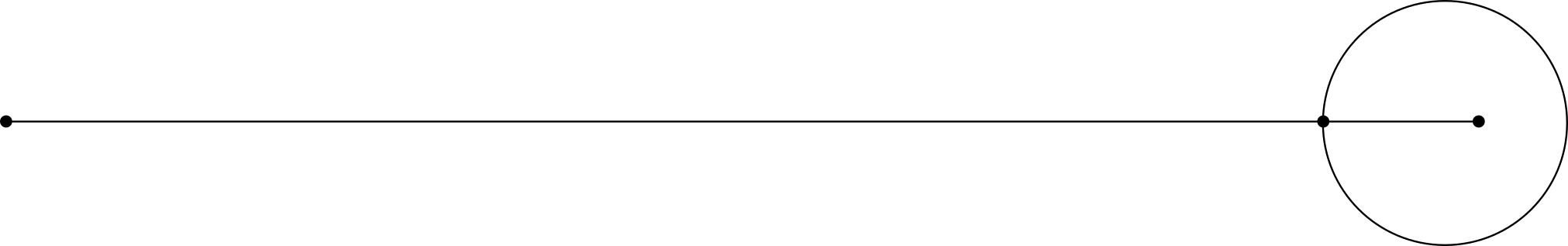}\hspace*{0em}}
        \caption{}
        \label{fig:graph simple alonge}
    \end{subfigure}
    
    \vskip\baselineskip
    \begin{subfigure}[b]{0.3\textwidth}  
        \centering 
        \includegraphics[height=.09\textheight]{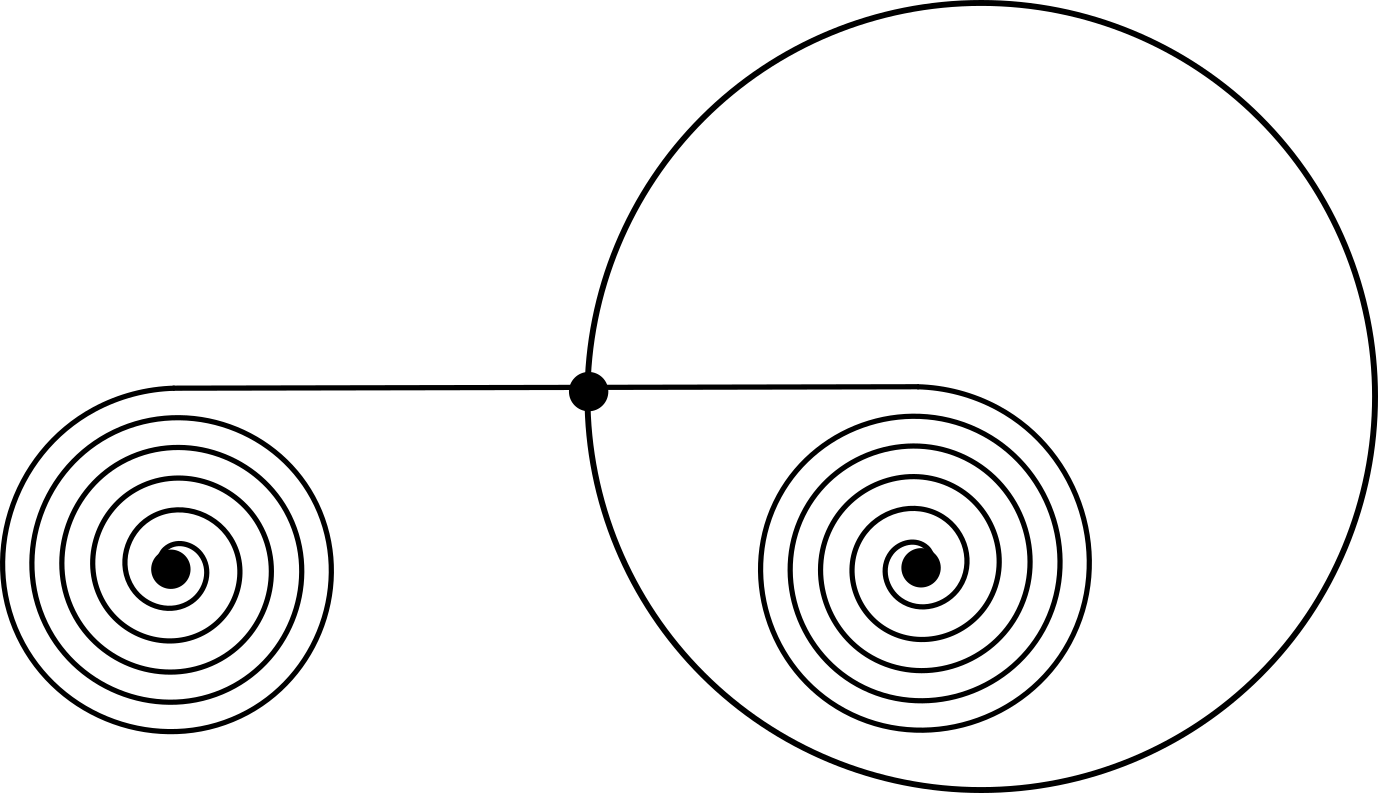}
        \caption{}
        \label{fig:graphe simple enroule}
    \end{subfigure}
\caption{Two graphs described in the introduction.}
    \label{fig:graph simple}
\end{figure}

To circumvent this issue, one possibility is to roll the line over itself, as in Figure~\ref{fig:graphe simple enroule}. Then, it becomes possible to slide the circle across the midpoint of the rolled line, without generating a second intersection.
Graph-like domains obtained from Figure~\ref{fig:graphe simple enroule} would thus be nice candidates for a Neumann \fdsh\/. Unfortunately they are 
not symmetric with respect to the $x$-axis. Indeed, recall that even if the counterexamples in Theorem~\ref{thm:resultat principal} need not be 
symmetric, symmetry is a key feature in our approach by means of \fdsh\/. In other words, one must construct a graph with properties
similar to that of Figure~\ref{fig:graphe simple enroule}, but which are symmetric with respect to the $x$-axis. To that end, we
append two additional spirals, ending up with a graph similar to that of Figure~\ref{fig:Steiner planaire}. Disregarding the loop, such
a graph belongs to the class known in the literature as caterpillar trees, namely trees such as all their nodes are adjacent to some
fixed path~\cite{harary-schwenk}. We will thus refer to it as a \emph{caterpillar tree with a loop} (see Definition~\ref{def:steiner}). The
study of caterpillar trees with a loop will occupy an important part of the present article. The goal is to find a configuration for which the nodal 
set of the second eigenfunction is located appropriately. We shall see that when the spirals are long compared to the other edges, the 
nodal set intersects the horizontal segment close to its midpoint. This result is formalised in Proposition~\ref{prop:construction du graph 
abstrait}. In particular, as the loop slides along the horizontal segment, it must cross that nodal point. This is precisely the property that was lacking in 
the graph of Figure~\ref{fig:graph simple alonge}.

\begin{figure}[h]
\centering
\includegraphics[scale=.8]{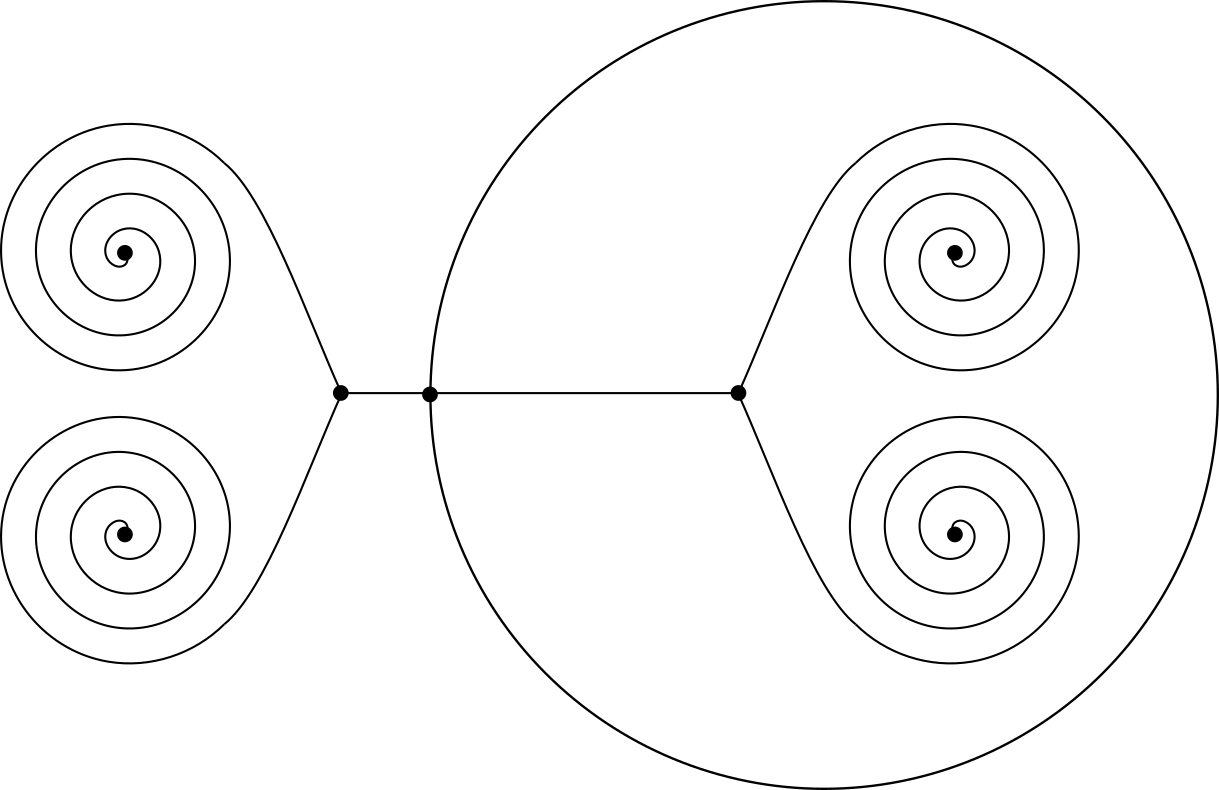}
\caption{A planar caterpillar tree with a loop.}
\label{fig:Steiner planaire}
\end{figure}

\begin{figure}[h]
    \centering
    \begin{subfigure}[b]{0.3\textwidth}
        
        \adjustbox{margin=0pt 0pt 0pt 0pt, center}{\includegraphics[height=.2\textheight]{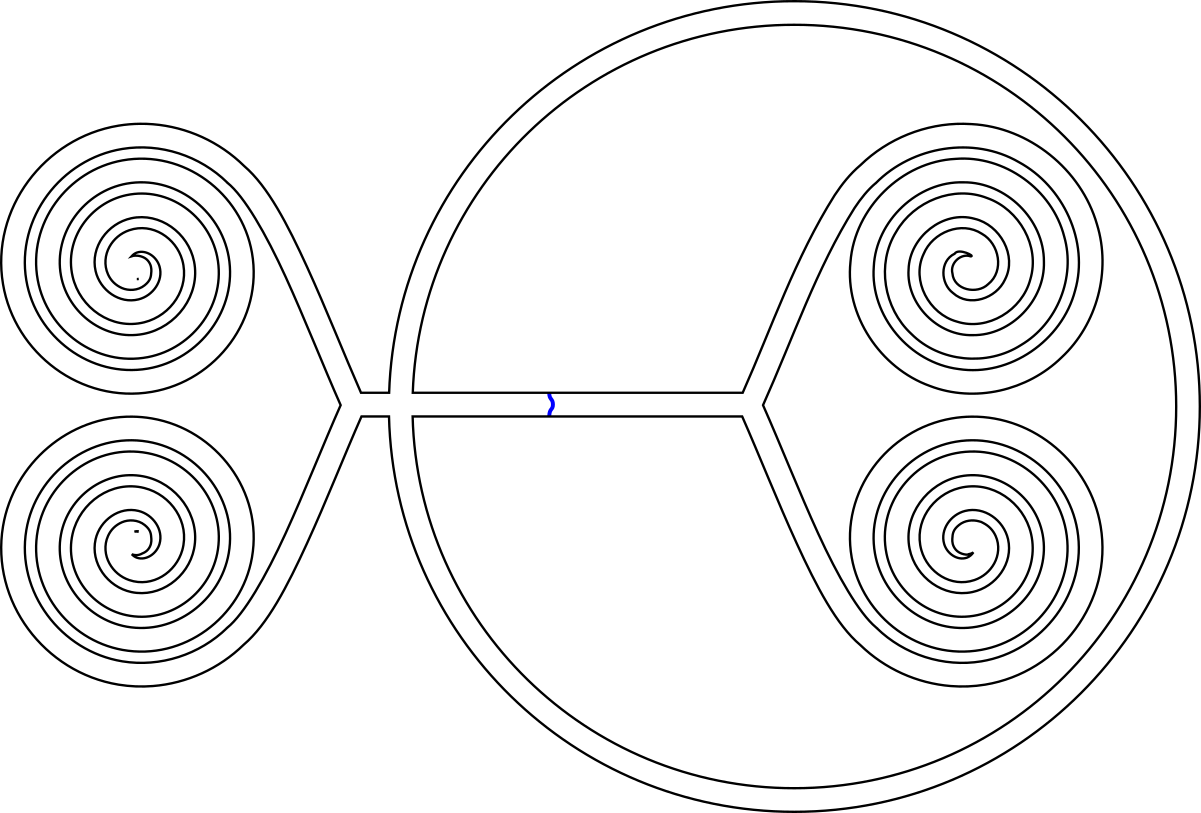}\hspace*{0em}}
        \caption{}
        \label{fig:domaine t-}
    \end{subfigure}
    
    \vskip\baselineskip
    \begin{subfigure}[b]{0.3\textwidth}  
        \centering 
        \adjustbox{margin=0pt 0pt 0pt 0pt, center}{\hspace*{2em}\includegraphics[height=.2\textheight]{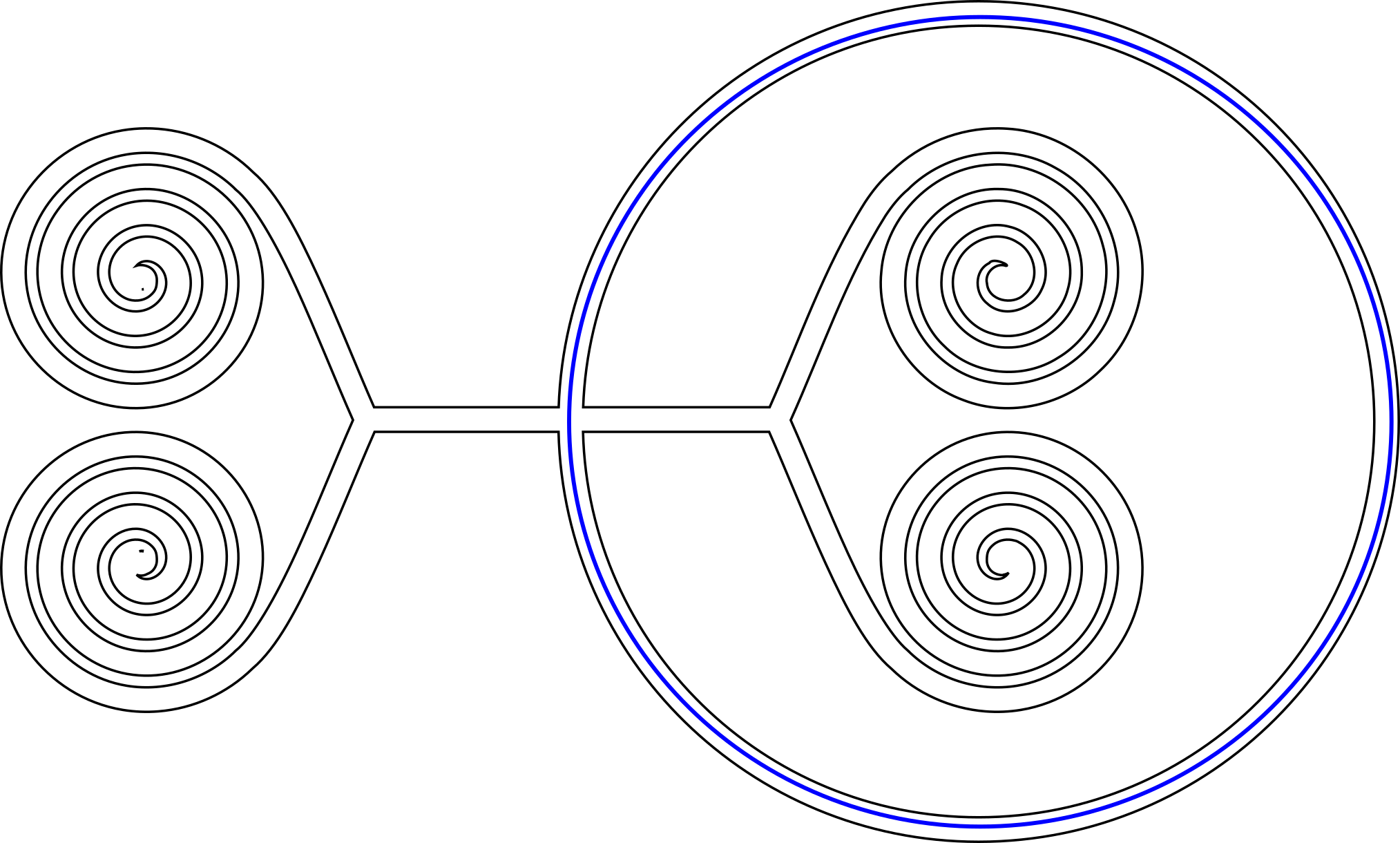}}
        \caption{}
        \label{fig:domaine t0}
    \end{subfigure}
    
    \vskip\baselineskip
    \begin{subfigure}[b]{0.3\textwidth}  
        \centering 
        \adjustbox{margin=0pt 0pt 0pt 0pt, center}{\hspace*{4em}\includegraphics[height=.2\textheight]{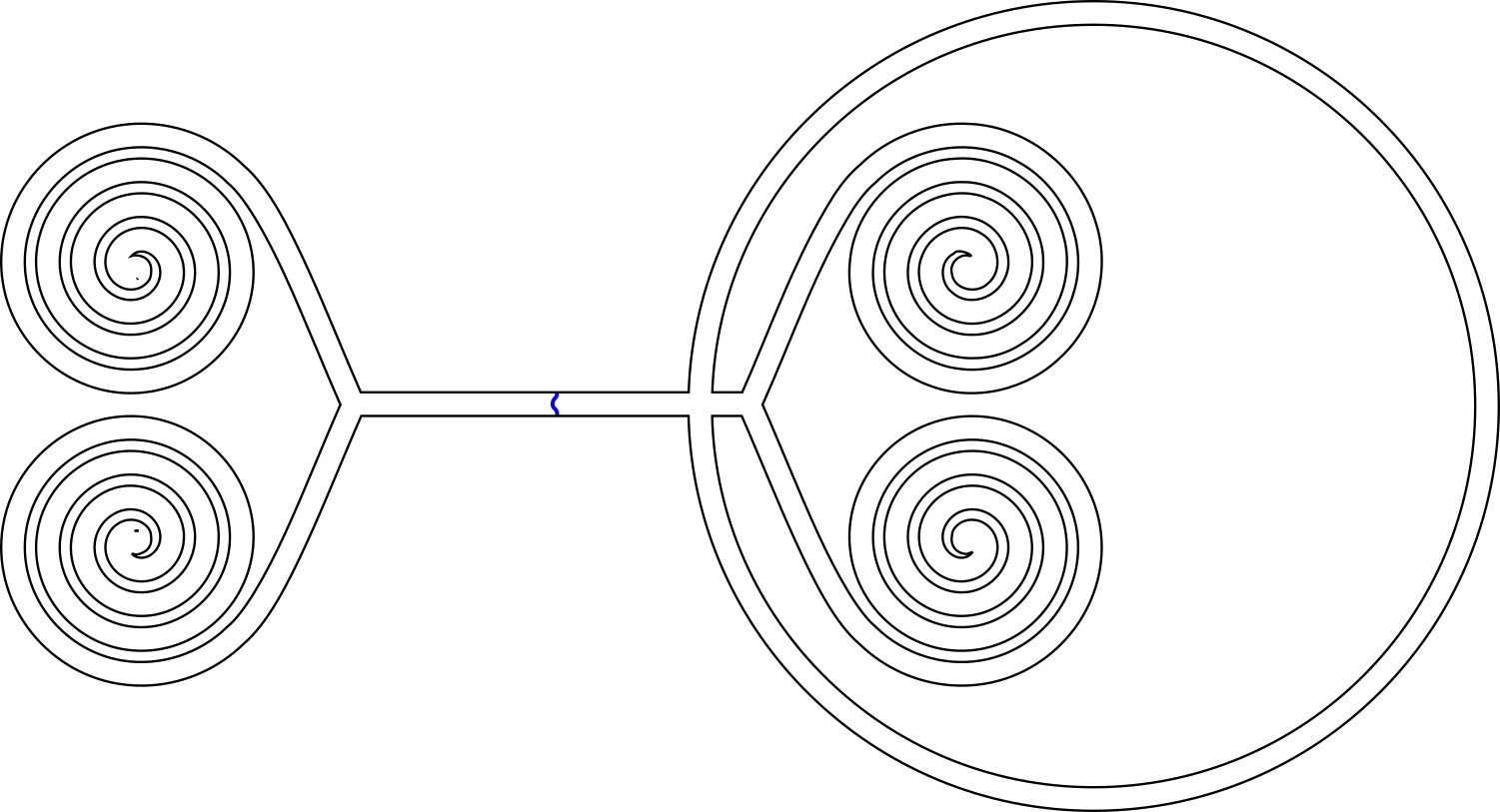}}
        \caption{}
        \label{fig:domaine t+}
    \end{subfigure}
\caption{Schematic representation of the Neumann \fdsh\/ described in the introduction. The expected nodal line is drawn in blue.}
    \label{fig:neumann fdsh}
\end{figure}

The above discussion justifies the consideration of graph-like domains close to planar caterpillar trees with a loop and very long
spirals, leading us to the family of domains sampled in Figure~\ref{fig:neumann fdsh}. As recalled in appendix~\ref{annexe:convergence}, when the width of those graph-like domains shrinks down to zero, the eigenvalues and eigenfunctions
converge to those of the limiting graph. Let us point out that the curvature of the spirals has no influence on the operator in
the limit, since Neumann boundary conditions are considered here~\cite[section~7.5]{berkolaiko-kuchment}. As a consequence, we expect the nodal line of any second eigenfunction on those
domains (in blue in Figure~\ref{fig:neumann fdsh}) to intersect the static part close to the vertical segment bisecting it. This would show points 
\eqref{it:signe segment} and \eqref{it:t- t+} in Definition~\ref{def:sliding domains}, which are the main issues in the construction of a
Neumann \fdsh.

To actually locate the nodal line, one needs refined convergence results for the eigenfunctions, when the width of the graph-like
domains tends to zero. As recalled in appendix~\ref{annexe:convergence}, an $L^2-$convergence result had already been obtained in~\cite{post}.
However, this is not enough to conclude that the nodal line is close to the nodal set of the limit graph. To solve this
problem, we derived what we called a \emph{strong transversal convergence} result -- see Proposition~\ref{prop:convergence} for more details --
which we believe to be of independent interest. In essence, this result indicates that the convergence holds on any transversal
slice of the graph-like domains, in a strong sense. The proof is based on elliptic estimates, but the fact that the domains shrink makes it non-trivial to apply
elliptic regularity theory, and explains why we only obtain convergence on slices. The transversal convergence allows to locate the nodal line
close to the vertical bisecting segment, as claimed above, and to eventually prove that graph-like domains of the type shown in Figure~\ref{fig:neumann fdsh}
do form a Neumann \fdsh.

\subsection*{Another possible strategy}
To prove Theorem~\ref{thm:resultat principal}, one could alternatively be tempted to use more directly the Dirichlet doubly-connected counterexample $\Omega$ from~\cite{freitas-leylekian}, and to argue in the following way. By the penultimate paragraph
in~\cite[p.3]{freitas-leylekian} each nodal domain $\Omega_\pm$ is in itself a doubly-connected shape, where the second Dirichlet eigenfunction $u_2$ is of one sign and vanishes on both boundary components of $\partial\Omega_\pm$. By a result of Weinberger's~\cite{weinberger}, there exists a piecewise analytic line $\gamma_\pm$  closed inside $\Omega_\pm$ without touching its boundary and where $\partial_nu_2=0$. Now the
lines $\gamma_+$ and $\gamma_-$ can be seen as the boundary components of a doubly-connected subdomain $\omega$ of $\Omega$. Furthermore,
the nodal line of $u_2$ is a Jordan curve inside $\omega$. This shows that $v:=u_2|_\omega$ is a Neumann eigenfunction over $\omega$ whose nodal line is closed and does not touch the boundary. The problem with this elementary argument is that one cannot exclude that $v$ is associated with
an eigenvalue of $\omega$ whose order is higher than two. Thus it only proves that there exists a Neumann eigenfunction
over $\omega$ whose nodal line is a Jordan curve.

\begin{figure}[h!]
\centering
\includegraphics[scale=.6]{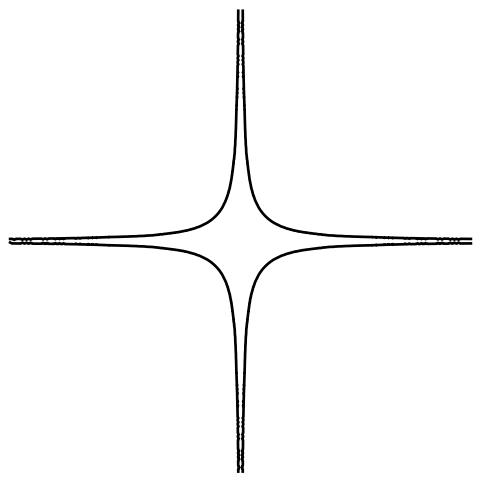}
\caption{Unbounded domain for which there exists a first non-trivial
Neumann eigenfunction whose nodal line
does not touch the boundary.}
\label{fig:unbounded}
\end{figure}

\subsection*{A note on unbounded domains} The above considerations were only concerned with bounded domains. However, the case of unbounded
domains is also of some interest since for instance, the idea for the counterexample presented in~\cite{freitas-leylekian} has its roots in the
unbounded counterexample first given in~\cite{freitas-krejcirik07}. Although the idea for the examples in the present paper now stemmed directly
from those provided in~\cite{freitas-leylekian}, we shall briefly discuss how such unbounded domains with a nodal line which does not touch the
boundary may also be constructed in the Neumann case. To this end, we consider a domain invariant by rotations of $\pi/2$ around a centre-point,
and with four infinite {\it arms} attached to a central body as shown in Figure~\ref{fig:unbounded}. These arms thin out sufficiently fast for the
necessary and sufficient condition for the Neumann Laplacian to have compact resolvent given in~\cite{evha87} to hold.

Clearly either the second eigenvalue is simple or it has a higher multiplicity. In the former case, an argument based on Courant's nodal
domain theorem and the rotational symmetry of the domain imply that the corresponding nodal line is, in fact, bounded and closed. In the
latter case, we will have two linearly independent eigenfunctions which are symmetric about the horizontal and vertical axes. The corresponding 
nodal lines for these two eigenfunctions will then coincide with either the vertical or the horizontal axes of symmetry of the domain, see the discussion in~\cite{kennedy18} for instance; it will thus be unbounded and not touch the boundary.

By carrying out a more careful and precise analysis it would be possible to determine which of the two cases occurs for a given domain, and also
to study what may happen for domains with only two infinite arms, closer to the examples in~\cite{freitas-krejcirik07}. However, our purpose
here is only to give an idea as to how the case of unbounded domains might be handled, and we will not pursue further a more detailed analysis in 
this case.

\subsection*{Open problem} We conclude the introduction by stating the following open problem.
\begin{center}
\emph{How small can the hole in a counterexample be?}
\end{center}
We mention that, unlike in the Dirichlet case, we are unable to argue that the hole may be reduced to a fracture. Indeed, to do so one would need to consider graph-like domains as in Figure~\ref{fig:neumann fdsh}, but where the handle now adheres to the spirals. The issue with this would be that the handle would become very long, while the above construction was precisely made for the handle to remain short compared to the spirals.

\section{Family of domains with a sliding handle}

In this section we introduce and study the notion of Neumann FDSH. Let us stress once again that its definition below is nothing else but the transcription of \cite[Definition~2.1]{freitas-leylekian}, where the concept of Dirichlet FDSH was settled. The only difference is that here, eigenvalues and eigenfunctions refer to the Neumann problem.

\begin{definition}[Neumann~\fdsh]\label{def:sliding domains}
Let $T$ be a non-degenerate interval. A one-parameter family of bounded connected open subsets $(\Omega_t)_{t\in T}$ of\/ $\R^2$ is called a Neumann family of domains with a sliding handle (Neumann \fdsh\/), whenever:
\begin{enumerate}
\item $\forall t\in T$, $\Omega_t$ is symmetric with respect to the $x$-axis and diffeomorphic to $\dd$.
\item $\forall t\in T$, the second eigenvalue of\/ $\Omega_t$ is simple.
\item\label{it:signe segment} $\forall t\in T$, the second eigenfunction of\/ $\Omega_t$ changes sign along a segment $\sigma_t\subseteq\Omega_t$ of the $x$-axis.
\item The eigenfunctions are uniformly continuous on $t$, in the sense that for each $t\in T$ there exists a second eigenfunction $u_t$ on $\Omega_t$ and a diffeomorphism $\Phi_t:\dd\to\Omega_t$ preserving the symmetry with respect to the $x$-axis, for which the pulled-back eigenfunction $v_t:=u_t\circ\Phi_t$ is continuous with respect to $t\in T$ uniformly with respect to $x\in \overline{\dd}$. In other words,
$T\ni t\mapsto v_t\in C(\overline{\dd})$
is continuous.
\item\label{it:t- t+} There exists $t_-\in T$ and $t_+\in T$ such that the nodal line of $v_{t_-}$ does not touch the outer boundary of $\dd$ and the nodal line of $v_{t_+}$ does not touch the inner boundary of $\dd$.
\end{enumerate}
\end{definition}

Then, we establish the main property of Neumann FDSH, according to which one of the domains in the family must admit eigenfunctions with a closed nodal line.

\begin{theoreme}[main property of FDSH]\label{thm:ligne nodale famille doublement connexe}
Let $(\Omega_t)_{t\in T}$ be a Neumann \fdsh. Then, there exists $t_0\in T$ such that the nodal line of a second eigenfunction in $\Omega_{t_0}$ does not touch the boundary, and hence is a closed Jordan curve inside $\Omega_{t_0}$.
\end{theoreme}

As already explained, this result is the analog of \cite[Theorem~2.2]{freitas-leylekian}, and it can be proved in a very similar way. For completeness, we sketch below the main arguments of the proof. The only precaution to take is to make sure that second eigenfunctions cannot have double zeros, like in the Dirichlet case. In the Neumann case this follows from~\cite[Proposition~2.8]{hoffmann-ostenhof-michor-nadirashvili}.

\begin{proof}
Let us assume by contradiction that the nodal line of the second eigenfunction $u_t$ over $\Omega_t$ touches the boudary for all $t\in T$, and set $\mathcal{I}_k:=\{t\in T:\dist(\mathcal{N}_t,\partial B_k)=0)\}$, for $k=1,2$. Recall that $v_t$ denotes the pull-back of $u_t$ on the annulus $\mathcal{A}$ and $\mathcal{N}_t$ is its nodal line. Our assumption shows that $\mathcal{I}_1$ and $\mathcal{I}_2$ cover $T$. By~\cite[Proposition~2.8]{hoffmann-ostenhof-michor-nadirashvili} $u_t$ and hence $v_t$ have no double zero, therefore we apply \cite[Lemma~3.1]{freitas-leylekian} to show that $\mathcal{I}_1$ and $\mathcal{I}_2$ are disjoint. On the other hand, by \cite[Lemma~3.2]{freitas-leylekian}, $\mathcal{I}_1$ and $\mathcal{I}_2$ are closed subsets of $T$. In other words $(\mathcal{I}_1,\mathcal{I}_2)$ forms a partition of the interval $T$ by closed subsets, which is a contradiction. 
\end{proof}

We now turn our attention to the construction a family of domains satisfying all the requirements listed in Definition~\ref{def:sliding domains}. The next proposition ensures that this construction shall be accomplished.

\begin{proposition}\label{prop:FDSH}
There exists a Neumann \fdsh.
\end{proposition}

The proof of the proposition will occupy the rest of the paper. As explained in the introduction, we will first need to study the spectrum of a particular class of metric graphs, as in section~\ref{sec:graph with a loop}. Then, we will construct a two-parameter family of graph-like domains, from which the FDSH will eventually be extracted. The proof of Proposition~\ref{prop:FDSH} strictly speaking can be found at the end of section~\ref{sec:graph-like domains}. Pretending that it is established, let us complete the proof of the main result of the paper.

\begin{proof}[Proof of Theorem~\ref{thm:resultat principal}]
By Proposition~\ref{prop:FDSH} there exists a Neumann FDSH, and by Theorem~\ref{thm:ligne nodale famille doublement connexe} the second eigenfunction on one of those domains admits a closed nodal line which does not touch the boundary.
\end{proof}

\section{Caterpillar trees with a loop}\label{sec:graph with a loop}

In this section, we define the notion of caterpillar tree with a loop and analyse the spectrum of the Neumann-Kirchhoff Laplacian on such graphs. The core part of the section is Proposition~\ref{prop:construction du graph abstrait}, which is a result on the location of the nodal set of the second eigenfunction when the rays of the caterpillar tree (see the definition below) are very long compared to the other edges.

\begin{definition}[Caterpillar tree with a loop]\label{def:steiner}
A metric graph $G$ is called a caterpillar tree with a loop if\/ $G$ admits the seven vertices $W,J,E,NW,SW,NE,SE$ and the seven edges $R_1$,$R_2$,$R_3$,$R_4$, $A$,$B$,$L$ where:
\begin{itemize}
\item $R_1,R_2,R_3,R_4$ are of length $R>0$ and link respectively $NW,SW,NE,SE$ to $W,W,E,E$.
\item $A$ is of length $a>0$ and links $W$ to $J$ while $B$ is of length $b>0$ and links $E$ to $J$.
\item $L$ is a loop around $J$, of length $l\geq 0$.
\end{itemize}
See Figure~\ref{fig:Steiner}. The vertex $J$ is called the junction point of the graph, while $W$ and $E$ are respectively called the west and east bifurcation points. The edges $R_1,R_2,R_3,R_4$ are called the rays of the graph. Lastly we define the bridge of $G$ as the concatenation of the edges $A$ and $B$ and denote its length $c=a+b$. Note that we authorize $l=0$, which amounts to removing the loop from the graph.
\end{definition}

\begin{figure}
\centering
\includegraphics[scale=.8]{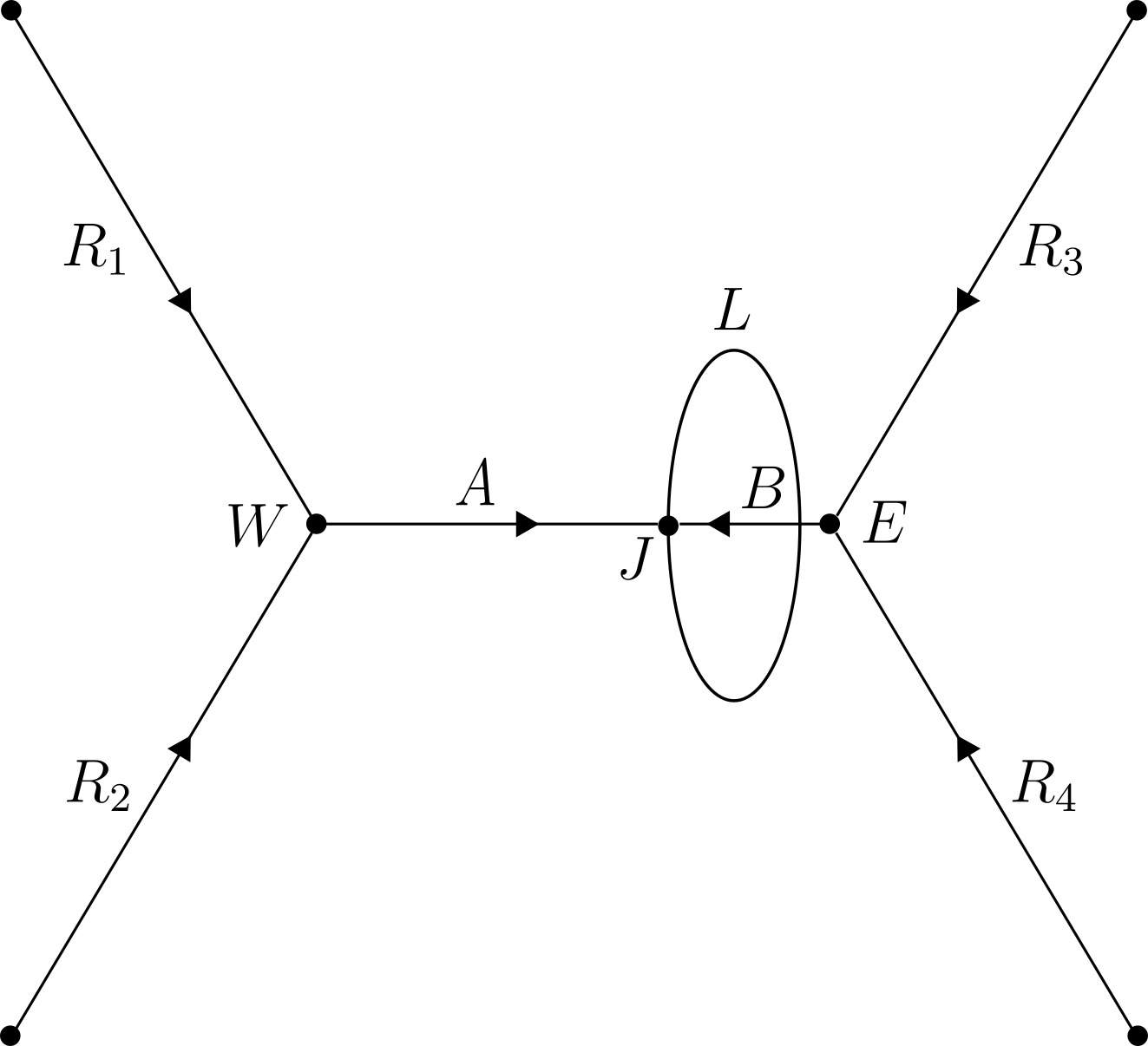}
\caption{A caterpillar tree with a loop. The black dots are the vertices of the graph. The arrows correspond to the orientation induced by the local parametrisation~\eqref{eq:parametrisation} of the graph.}
\label{fig:Steiner}
\end{figure}

Let $G$ be a caterpillar tree with a loop. In what follows, unless otherwise mentioned we will use the following local arclength parametrisation of the graph:

\vspace{-.5\baselineskip}
\begin{equation}\label{eq:parametrisation}
\begin{minipage}{.9\textwidth}
\begin{itemize}[leftmargin=.5em]
\item $A$ is parametrised by $t\in [0,a]$ in a way that $J$ corresponds to $t=a$;
\item $B$ is parametrised by $t\in [0,b]$ in a way that $J$ corresponds to $t=b$;
\item $L$ is parametrised by $t\in [0,l]$ in a way that $J$ corresponds to $t=0,l$;
\item The rays are parametrised by $t\in[0,R]$ in a way that $t=R$ corresponds to $W$ on $R_1,R_2$ and to $E$ on $R_3,R_4$;
\item The bridge is parametrised by $t\in [-c/2,c/2]$ in a way that $W$ corresponds to $t=-c/2$ and $E$ corresponds to $t=+c/2$.
\end{itemize}
\end{minipage}
\end{equation}
\vspace{-0\baselineskip}
This parametrisation naturally endows the edges $A,B,L,R_1,...,R_4$ with an orientation, which is represented by the arrows in Figure~\ref{fig:Steiner}, for the reader's convenience. Note that $G$ is symmetric with respect to the vertical reflection $r_{NS}$. Furthermore, when $G$ has no loop, it is also symmetric with respect to the horizontal reflection $r_{EW}$. Those reflections are the involutions respectively defined by
\begin{align}\label{eq:NS}
r_{NS}= & \left(
\begin{array}{rcl}
t\in [0,R]\cong R_i & \mapsto & t\in[0,R]\cong R_{i-(-1)^i},\quad i=1,...,4 \eqskip
t\in A\cup B & \mapsto & t\in A\cup B \eqskip
t\in [0,l]\cong L & \mapsto & l-t\in [0,l]\cong L
\end{array}
\right)
\eqskip
\label{eq:EW}
r_{EW}= &\left(
\begin{array}{rcl}
t\in [0,R]\cong R_i & \mapsto & t\in[0,R]\cong R_{(i+2)\hspace*{-.5em}\mod 4},\quad i=1,...,4 \eqskip
t\in [-c/2,c/2]\cong A\cup B & \mapsto & c/2-t\in[-c/2,c/2]\cong A\cup B
\end{array}
\right)
\end{align}

We now endow the graph $G$ with the Neumann-Kirchhoff Laplacian, defined as follows. The domain of this operator is the set $H^2(G)$ of functions $u$ whose restriction to each edge $e$, when identified with an open segment $I\subseteq\R$ of same length as $e$, is in $H^2(I)\subseteq C^1(\overline{I})$, and such that at a vertex $v$ the following Kirchhoff condition is fulfilled:
\begin{equation}\label{eq:kirchhoff}
\begin{cases}
{\ds \lim_{e\ni x\to v}u(x)}=u(v) & \forall e\sim v,\eqskip
\dsum_{e\sim v} \partial_n (u|_e)(v)=0.
\end{cases}
\end{equation}
Here, for a given edge $e$, $e\sim v$ means that $v$ is an endpoint of $e$, while $\partial_n$ denotes the derivative at the vertex in the direction pointing outward the edge. Note that when the vertex $v$ is the endpoint of a single edge, we get back the traditional Neumann condition $\partial_n u(v)=0$. Over the functional space $H^2(G)$, the Laplacian $-\Delta$ is defined by the expected formula $-\Delta u(x):=-u|_e''(x)$ for any $x\in e$. This defines a self-adjoint operator on the graph, which admits a sequence of non-negative eigenvalues $\nu^2$, growing to $+\infty$. The corresponding eigenfunctions $u\in H^2(G)$ are given by the relation $-\Delta u =\nu^2 u$ and they may be expressed on each edge as combination of trigonometric functions. In particular, for a caterpillar tree with a loop, any eigenfunction $u$ is of the form
\begin{equation}\label{eq:fonction propre expression}
u(t)=
\begin{cases}
\alpha_1\cos(\nu t) & \text{if }t\in [0,R]\cong R_1\\
\alpha_2\cos(\nu t) & \text{if }t\in [0,R]\cong R_2\\
\alpha_3\cos(\nu t) & \text{if }t\in [0,R]\cong R_3\\
\alpha_4\cos(\nu t) & \text{if }t\in [0,R]\cong R_4\\
\gamma_1\cos(\nu t)+\gamma_2\sin(\nu t) & \text{if }t\in[0,a]\cong A\\
\delta_1\cos(\nu t)+\delta_2\sin(\nu t) & \text{if }t\in[0,b]\cong B\\
\lambda_1\cos(\nu t)+\lambda_2\sin(\nu t) & \text{if }t\in[0,l]\cong L,
\end{cases}
\end{equation}
for some coefficients $\alpha_1,...,\alpha_4,\gamma_1,\gamma_2,\delta_1,\delta_2,\lambda_1,\lambda_2\in\R$, $i=1,...,4$. The eigenfunction satisfies the Kirchhoff condition~\eqref{eq:kirchhoff} if and only if the vector $(
\alpha_1, \alpha_2, \alpha_3, \alpha_4, \gamma_1,\gamma_2,\delta_1,\delta_2,\lambda_1,\lambda_2)$ is in the kernel of the matrix
\begin{gather*}\label{eq:matrice}
M=\begin{pmatrix}
\cos(\nu R) & -\cos(\nu R) &  &  &  &  &  &  &  &  \\
\cos(\nu R) &  &  &  & -1 &  &  &  &  & \\
\sin(\nu R) & \sin(\nu R) &  &  &  & 1 &  &  &  & \\
&  & \cos(\nu R) & -\cos(\nu R) &  &  &  &  &  &  \\
&  & \cos(\nu R) &  &  &  & -1 &  &  & \\
&  & \sin(\nu R) & \sin(\nu R) &  &  &  & 1 &  & \\
&  &  &  & \cos(\nu a) & \sin(\nu a) & -\cos(\nu b) & -\sin(\nu b) &  & \\
&  &  &  & \sin(\nu a) & -\cos(\nu a) & \sin(\nu b) & -\cos(\nu b) & \sin(\nu l) & 1-\cos(\nu l) \\
&  &  &  & \cos(\nu a) & \sin(\nu a) & & & -1 & \\
&  &  &  &  &  &  &  & \cos(\nu l)-1 & \sin(\nu l) \\
\end{pmatrix}
\end{gather*}

Therefore the number $\nu^2$ is an eigenvalue if and only if the above matrix is not invertible. By computing the determinant, we see that this is equivalent to the condition 
\begin{equation}\label{eq:determinant}
\cos(\nu R)^2\sin(\nu l/2)f(\nu,R,a,b,l)=0
\end{equation}
where
\begin{equation}\label{eq:determinant fonction}
\begin{split}
f(\nu,R,a,b,l):= & 10\cos(\nu(a-b))\sin(\nu l/2)-6\sin(\nu(a+b+l/2))+\sin(\nu(a+b+l/2-2R)) \\
&-6\cos(\nu(a-b))\cos(2\nu R)\sin(\nu l/2)+9\sin(\nu(a+b+l/2+2R)).
\end{split}
\end{equation}
When $l=0$, the last two lines and columns of the matrix must be removed, as well as the factor $\sin(\nu l/2)$ in equation~\eqref{eq:determinant}. Observe that $\nu=2k\pi/l$ and $\nu=(2k+1)\pi/(2R)$ are always roots of this equation, for all $k\in\N$. They correspond to eigenfunctions localised respectively on the loop and on one of the paths $R_1\cup R_2$ and $R_3\cup R_4$. The eigenvalue $\nu^2=0$ is associated with constant functions, which are actually eigenfunctions on any (compact) metric graph.

The second eigenvalue is the square of the first positive zero of the previous equation. The coefficients $(\alpha_1,...,\alpha_4,\gamma_1,\gamma_2,\delta_1,\delta_2,\lambda_1,\lambda_2)$ in the expression of the associated eigenfunctions are obtained as elements in the kernel of $M$. But it is also uselful to recall that the second eigenvalue shall alternatively be computed thanks to the following min-max formula \cite[equation (2.6)]{berkolaio-kennedy-kurasov-mugnolo}
\begin{equation}\label{eq:minmax}
\nu^2=\inf\left\{\frac{\int_G |v'|^2}{\int_G v^2}:v\in H^1(G),\quad \int_G v=0\right\},
\end{equation}
where $H^1(G)$ is the set of functions in $\oplus_e H^1(e)$ which are continuous at each vertex, i.e. which satisfy the first condition in~\eqref{eq:kirchhoff}. This property allows for instance to prove that the second eigenvalue on a caterpillar tree is smaller than the second eigenvalue on the subgraph obtained as the concatenation of $R_1$ and $R_2$ (or $R_3$ and $R_4$).

\begin{lemme}\label{lemme:borne sup valeur propre}
The second eigenvalue $\nu^2$ of a caterpillar tree with a loop satisfies $\nu^2<\pi^2/(2R)^2$.
\end{lemme}

\begin{proof}
Let $G$ be a caterpillar tree with a loop, and assume without loss of generality that $a\leq b$. Let $\mu=\pi/(2(R+a))$ and consider the function $w$ such that, using the parametrisation~\eqref{eq:parametrisation}, $w(t)=\cos(\mu t)$ for $t \in R_1$ and for $t\in R_2$, whereas $w(t)=\cos(\mu(t+R))$ for $t\in A$. When extended by zero on the graph $G^*$ obtained from $G$ by removing the loop $w$ is a function in $H^1(G^*)$. Its horizontal reflection $\tilde{w}$ is also in $H^1(G^*)$ and since $a\leq b$ those two functions are not simultaneously non-zero. As a result, taking as a test function $v$ on $G$ the extension by zero of $w-\tilde{w}$, we get
$$
\frac{\int_{G}|v'|^2}{\int_{G}v^2}=\mu^2.
$$
Furthermore $v$ has zero mean value by construction, and hence the first eigenvalue over $G$ is not greater than $\mu^2<\pi^2/(2R)^2$.
\end{proof}

The previous lemma will serve multiple times in the subsequent lines. As an example, it can be used to prove the following central remark according to which the second eigenvalue is simple, as long as the rays are elongated enough.

\begin{lemme}\label{lemme:simplicité}
Let $G$ be a caterpillar tree with a loop such that $l\leq 2R$. Then the second eigenvalue is simple.
\end{lemme}

\begin{remarque}\label{rmq:symétrie fonction propre graphe}
An inspection of the proof shows that the associated eigenfunction is symmetric with respect to vertical reflections and that it does not vanish identically on any ray. Actually, by symmetry and by Courant's principle it cannot vanish at all inside a ray. 
\end{remarque}

\begin{proof}
Let us first establish the symmetry property. Any eigenfunction $u$ can be decomposed as the sum of a symmetric and an antisymmetric eigenfunction (with
respect to vertical reflections). If the antisymmetric part were non-trivial, either its restriction to the concatenation of the edges $R_1$ and $R_2$ or $R_3$ and $R_4$ would be an eigenfunction there (indeed, the Kirchhoff condition would be satisfied at $W$ and $E$); or its restriction to the loop, if any, would be an eigenfunction of the \emph{Dirichlet} Laplacian. In the former case $\nu^2$ would be greater or equal than $\mu^2:=\pi^2/(2R)^2$, and in the latter it would be greater or equal than $\pi^2/l^2\geq\mu^2$, a contradiction with Lemma~\ref{lemme:borne sup valeur propre} in both cases. Therefore $u$ is symmetric.

Now we prove the simplicity. Observe first that $u$ cannot vanish identically on a ray of $G$. Indeed, otherwise by symmetry it would need to vanish on the reflection of the ray and hence on the corresponding bifurcation point, say $W$. In this case, by the Kirchhoff law the normal derivative of $u|_A$ would also vanish at $W$. By Holmgrem's theorem, the nullset of $u$ would then propagate over $A$ up to the junction point $J$.

If $l>0$ and if $u$ were not identically zero on $L$, the restriction of $u$ to the loop would be an eigenfunction of the \emph{Dirichlet} Laplacian, contradicting Lemma~\ref{lemme:borne sup valeur propre} as above.

If on the contrary $u$ vanishes on the loop or if $l=0$, the nullset of $u$ now propagates up to the bifurcation point $E$. As a result the restriction of $u$ to the concatenation of $R_3$ and $R_4$ is an eigenfunction, which is again a contradiction with Lemma~\ref{lemme:borne sup valeur propre}.

The conclusion of this discussion is that a second eigenfunction cannot vanish identically on a ray of $G$. As a result, if there were two independent eigenfunctions associated with $\lambda$, these eigenfunctions would be colinear on each ray by~\eqref{eq:fonction propre expression}, hence one could construct a third eigenfunction vanishing on a ray, a contradiction.
\end{proof}

In the next result, we prove that up to elongating enough the rays in $G$, we can locate the nodal set of the second eigenfunction in an arbitrary neighbourhood of the midpoint of the bridge of $G$. Furthermore, the result is uniform with respect to the position of the junction point.

\begin{proposition}[Localisation of the nodal set]\label{prop:construction du graph abstrait}
Let $G$ be a caterpillar tree with a loop of length $l=4\pi$ and a bridge $P$ of length $c=2$. For any $\epsilon>0$ there exists $R_0>0$ such that if the rays of $G$ are of length greater than $R_0$, the second eigenvalue is simple and the associated eigenfunction has a nodal set whose intersection with $P$ is nonempty and lies in $(-\epsilon,\epsilon)\subseteq [-1,1]\cong P$, independently of the position of the junction point over $P$.
\end{proposition}

\begin{remarque}\label{remarque:simplicité}
The result holds with other values for the length and of the loop and of the bridge, but for simplicity we stick to this configuration. The reason for chosing these particular values will be made clearer in the next section, when the graph will be embedded in the plane.
\end{remarque}

\begin{proof}
The simplicity is given by Lemma~\ref{lemme:simplicité}, provided that $R_0\geq 2\pi$. Let $x\in(-1,1)\cong P$ denote the position of the junction point $J$. To prove the result, we rescale $G$ in such a way that the length of the rays is equal to one. In this setting, the length of the loop is now $l=2\pi r$, that of the bridge is $c=r$, while $a=r(1+x)/2$ and $b=r(1-x)/2$, for some small parameter $r>0$. In the next lines, we will analyse the asymptotic behaviour of $u$ on the bridge of $G$ in the limit $r\to0$. To do so, the first step is to understand the behaviour of the root $\nu$.

By \cite{berkolaiko-latushkin-sukhtaiev}, $\nu(r)$ converges to $\pi/2=:\nu(0)$, which corresponds to the second eigenvalue obtained by formally making the loop and the bridge of $G$ collapse to a point. In particular $\sin(\nu l/2)\neq0$ when $r$ is close enough but not equal to zero. Furthermore by Lemma~\ref{lemme:borne sup valeur propre}, $\nu(r)$ is stricly smaller than $\nu(0)$. This means that $\cos(\nu)\neq0$ whenever $r>0$. We conclude that $\nu$ is a root of the function $f$ defined in~\eqref{eq:determinant fonction}. By the implicit function Theorem (since $\partial_\nu f(\nu(0),1,0,0,0)\neq0$), $\nu$ is an analytic function of $(r,x)$, for $r$ close to zero and $x$ in any compact subset of $(-1,1)$. By computing the derivative of $f(\nu(r),1,a(r),b(r),l(r))=0$ at $r=0$, we see that $\nu'(0)=-\nu(0) = -\pi/2$. As a result, $\nu(r)=\pi(1-r)/2+\so(r)$. By analyticity of $\nu$ in $(r,x)$, this expansion is uniform with respect to $x$.

Now let us address $u$. To that end, we denote by $\alpha_1,...,\alpha_4,\gamma_1,\gamma_2,\delta_1,\delta_2,\lambda_1,\lambda_2$ the coefficients in its expression (cf.~\eqref{eq:fonction propre expression}). Observe that by Remark~\ref{rmq:symétrie fonction propre graphe}, $u$ does not vanish on the rays $R_1,...,R_4$, hence the coefficients $\alpha_1,...,\alpha_4$ are non-zero. Furthermore, by applying the conditions in~\eqref{eq:kirchhoff} at the bifurcation points $W$ and $E$, we find that
$$
(\alpha_1,\alpha_2,\gamma_1,\gamma_2),(\alpha_3,\alpha_4,\delta_1,\delta_2)\in \text{Vect}(X(\nu)),\qquad X(\nu):=
\begin{pmatrix}
1 \eqskip 1 \eqskip \cos(\nu) \eqskip -2\sin(\nu)
\end{pmatrix}.
$$
In view of~\eqref{eq:fonction propre expression}, this means that the restrictions of $u$ to $A$ and $B$ are non-trivial multiples of $\cos(\nu)\cos(\nu t)-2\sin(\nu)\sin(\nu t)$ for $t\in[0,a]$ and $t\in[0,b]$, respectively, though with multiplicative constants that may differ on $A$ and $B$. We now rescale the bridge to its original size by introducing $v_A(t)=u|_A(r(1+t)/2)$ and $v_B(t)=u|_B(r(1-t)/2)$, which are defined for $t\in[-1,x]$ and $t\in[x,1]$. Using the expansion of $\nu$ previously obtained we compute, up to the multiplicative constants,
$$
v_A(t)=\frac{\pi rt}{2}+\so(r),\qquad v_B(t)=\frac{\pi rt}{2}+\so(r),
$$
with $\so(r)$ which are now uniform in $x$ and $t$. This shows that $v_A$ and $v_B$ cannot vanish outside $(-\epsilon,\epsilon)$, when $r$ is small enough regardless $x$. In other words, the eigenfunction of $G$ vanishes on the bridge at a single point close to its midpoint as $r$ is small, uniformly with respect to $x$. Since we already argued that it cannot vanish on the rays, this concludes the proof.
\end{proof}

\section{Planar graphs with a sliding loop}

The next step towards the construction of a Neumann \fdsh~is to embed the caterpillar tree $G$, constructed in Proposition~\ref{prop:construction du graph abstrait}, in the plane. A trivial way to do so would be to embed the loop in the open upper half plane, and to keep the embedding of the bridge in the $x$-axis. The issue is that the planar graph obtained in this way, and subsequently the graph-like domain, would not be symmetric with respect to the $x$-axis, in contradiction with Definition~\ref{def:sliding domains}.

Therefore, one must embed the loop across the $x$-axis. The problem with this construction is that as the loop is of length $4\pi$ while the rays are potentially very long, they will most probably intersect each other, in which case the planar graph is not isometric to $G$ anymore. As explained in the introduction the solution is to sharply roll the rays over themselves. One must also pay attention to the fact that the loop must be able to slide between the bifurcation points without intersecting the rolled rays. This is guaranteed by the next lemma (see also Figure~\ref{fig:Steiner planaire}). Note that the choice of $l=4\pi$ and $c=2$ as done in Proposition~\ref{prop:construction du graph abstrait} is important here, since then the bridge is as short as the radius of a circle isometric to the loop. This allows to consider such a circle as the embedding of the loop and to let it slide along the bridge without taking the risk of any additional intersection.

\begin{lemme}\label{lemme:aplanissement}
Let $R>0$. There exists a family $(G_t)_{t\in(-1,1)}$ of planar graphs symmetric with respect to the $x$-axis such that for each $t\in(-1,1)$, $G_t$ is a caterpillar tree with a loop of length $l=4\pi$, with a bridge of length $c=2$ and with rays of length $R$. Furthermore, the subgraph obtained from $G_t$ by removing the loop is independent of\/ $t$ and its bridge is the line segment $[(-1,0),(1,0)]$. Lastly, the loop of $G_t$ is of the form $\mathcal{C}+(t,0)$, where $\mathcal{C}$ is a Jordan curve passing through the origin.
\end{lemme}

\begin{proof}
The intersection of the set $Q=\{(x,y)\in\R^2, x>1, y>0\}$ with the open disks $D_t$ of radius two centered at $(3+t,0)$, namely
$$
Q\cap\bigcap_{-1<t<1}D_t
$$
is a convex planar set with nonempty interior. The point $E=(1,0)$ belongs to its closure. Therefore it contains a path $R_3$ of length $R$ with one of its endpoints being $E$. Reflect $R_3$ with respect to the $x$-axis and the $y$-axis to generate $R_4$ and $R_1$, respectively, and reflect $R_1$ with respect to the $x$-axis to define $R_2$. By construction the circle of radius two centered at $(3+t,0)$ never intersects $R_1,...,R_4$ when $t\in(-1,1)$. Hence the caterpillar tree $G_t$ with this circle as the loop, with the paths $R_1,...,R_4$ as the rays, and with the line segment $P=[(-1,0),(1,0)]$ as the bridge meets all the requirements.
\end{proof}

\begin{remarque}
It is clear from the proof that $G_t$ can be assumed to be smooth and non-degenerate, in the sense that each edge is a smooth line and that the edges arriving at a given vertex meet non-tangentially with respect to each other. We can even assume that the edges are locally straight near the vertex.
\end{remarque}

\newcommand{\fgsl}{\textup{FGSL}}
The next definition formalises the planar graphs exhibited in this section, as they are the main ingredient for constructing the Neumann \fdsh. We also mention that although not strictly necessary, it will be technically convenient to assume those graphs to be locally straight around each vertex. This assumption can be made without loss of generality, as already mentioned.

\begin{definition}
A family of graphs as in Lemma~\ref{lemme:aplanissement} is called a family of planar graphs with a sliding loop \textup{(\fgsl)}. We say that the family is smooth (resp. locally straight) if each graph is smooth (resp. locally straight around the vertices).
\end{definition}

\section{Graph-like domains with a sliding handle}\label{sec:graph-like domains}

In this section we construct a two-parameter family of domains obtained by thickening an \fgsl. One of the parameters is the sliding parameter $t\in(-1,1)$ while the other one controls the thickeness of the domains. We point out that the main purpose of the construction is that the resulting domains have their Neumann eigenvalues and eigenfunctions close to that of the graphs of the \fgsl, when the thickeness parameter is small enough. To that end, we will make use of the results in appendix~\ref{annexe:convergence}. In particular we need to ensure that the domains constructed converge to the graphs of the \fgsl\/ in the sense of Definition~\ref{def:convergence vers graph} as the thickeness parameter goes to zero. The domains also need to fulfil some additional symmetry and regularity properties that will be useful later. The existence of such a family of graph-like domains is guaranteed by the next proposition.

\newcommand{\A}{\mathcal{A}}
\begin{proposition}\label{prop:graph-like domains}
Let $(G_t)_{t\in(-1,1)}$ be a smooth and locally straight \fgsl. For any compact subset $T$ of $(-1,1)$, there exists a family of bounded open sets $(\Theta_{h,t})_{t\in T,0<h<h_0}$, for some $h_0>0$, such that:
\begin{enumerate}
\item For all $t\in T$ and all $0<h<h_0$, $\Theta_{h,t}$ is doubly connected and symmetric with respect to the $x$-axis.
\item For all $t\in T$, if $t_n\to t$ and $h_n\to 0$, $\Theta_{h_n,t_n}$ converges to $G_t$ in $C^\infty$, in the sense of Definition~\ref{def:convergence vers graph}.
\item For each $0<h<h_0$, the sets $(\Theta_{h,t})_{t\in T,0<h<h_0}$ are $C^\infty$ with a continuous and piecewise $C^1$ dependence on $t\in T$. This means that for all $t\in T$ there exists a $C^\infty$-diffeomorphisms $\Phi_{h,t}$ mapping the annulus $\A$ to $\Theta_{h,t}$, such that $(t\mapsto \Phi_{h,t},\Phi_{h,t}^{-1}\in C^\infty(\R^d,\R^d)^2)$ is continuous and piecewise $C^1$. Furthermore, $\Phi_{h,t}$ preserves the symmetry with respect to the $x$-axis.
\end{enumerate}
\end{proposition}

\begin{remarque}
We believe that this kind of result should hold for more general families of graphs. More precisely, we expect that for any one-parameter family of metric graphs depending in a sufficiently smooth way on the parameter, one should be able to contruct a family of domains satisfying the last two properties of the proposition. Yet as can be seen from the proof (especially Step~2) we were unable to find an argumentation that does not rely on the specific geometry of the \fgsl.
\end{remarque}

\begin{proof}
The proof follows three main steps.

\medskip
\noindent\emph{Step 1 (Construction of the domains for a fixed $t\in T$)}: In this step, we consider a fixed value $t\in T$. We will construct $\Theta_{h,t}$ from the planar graph $G_t$ by blocks, i.e. we will build separately the neighborhoods $U_{v,h}$ around each vertex $v$ of $G_t$ and the neighborhoods $U_{e,h}$ around each edge $e$ of $G_t$ involved in Definition~\ref{def:convergence vers graph}. Let us begin with the case of an edge $e$, following closely \cite[Section~3.1]{post}. We consider $\psi:[-l_e/2,l_e/2]\to \overline{e}\subseteq\R^2$ to be an arc-length parametrization of $e$. This means that $\psi'$ is a unit vector of $\R^2$. Note that $\psi$ is smooth by smoothness of $G_t$. We then define for $(x,y)\in [-l_e/2,l_e/2]\times[0,1]=\overline{U}_e$ the tubular transformation
$$
F_{h,e}(x,y)=\psi((1-\alpha h)x)+h(y-1/2)\vec{n}((1-\alpha h)x),
$$
where $\vec{n}=(\psi_2',-\psi_1')$ is a normal unit vector of $e$ and $\alpha>0$ is a security parameter (to be chosen later) ensuring that the neighborhhoods of the different edges do not interesect. When $h_0$ is small enough, this transformation is smooth and injective, by smoothness and injectivity of $\psi$. Furthermore, the matrix $(DF_{h,e})^TDF_{h,e}$ is
$$
\begin{pmatrix}
(1-\alpha h)^2[(\psi_1'+yh\psi_2'')^2+(\psi_2'-yh\psi_1'')^2] & h^2(1-\alpha h)y(\psi_1''+\psi_2'') \\
h^2(1-\alpha h)y(\psi_1''+\psi_2'') & h^2(\psi_1'^2+\psi_2'^2)
\end{pmatrix}
$$
which satisfies the asymptotic expansion of Definition~\ref{def:convergence vers graph} since $\psi_1'^2+\psi_2'^2=1$ (note that here and below, for readability we dropped the point $(1-\alpha h)x$ where the functions $\psi_i',\psi_i'',i=1,2$ are evaluated). Since $\psi$ is smooth the asymptotic expansion holds in $C^\infty(U_e)$. Thus, it is enough to consider $F_{h,e}(U_e)$ as the neighborhood $U_{e,h}$ of $e$. It now remains to construct the neighborhood $U_{v,h}$ for a vertex $v$ of $G_t$ and to make sure that we obtain a smooth domain by gluing it with the adjacent edge neighborhoods. If $d\in\{1,3,4\}$ is the degree of $v$ and $e_1,...,e_d$ are the adjacent edges, we consider the segments $\sigma_1(h),...,\sigma_d(h)$ given by $\sigma_i(h):=F_{e_i,h}(\{l_{e_i}/2\}\times (0,1))$, with the convention that the parametrisation $\psi$ of the edge $e_i$ used in the construction of $F_{e_i,h}$ is such that $\psi(l_{e_i}/2)=v$. Since $G_t$ is locally straight, $\sigma_i$ is an open segment of length $h$ whose midpoint $\psi((1-\alpha h)l_{e_i}/2)$ is at distance $\alpha h l_{e_i}/2$ from the vertex $v$, when $h$ is small (see Figure~\ref{fig:Uv}).

\begin{figure}[h]
\centering
\includegraphics[height=.40\textheight]{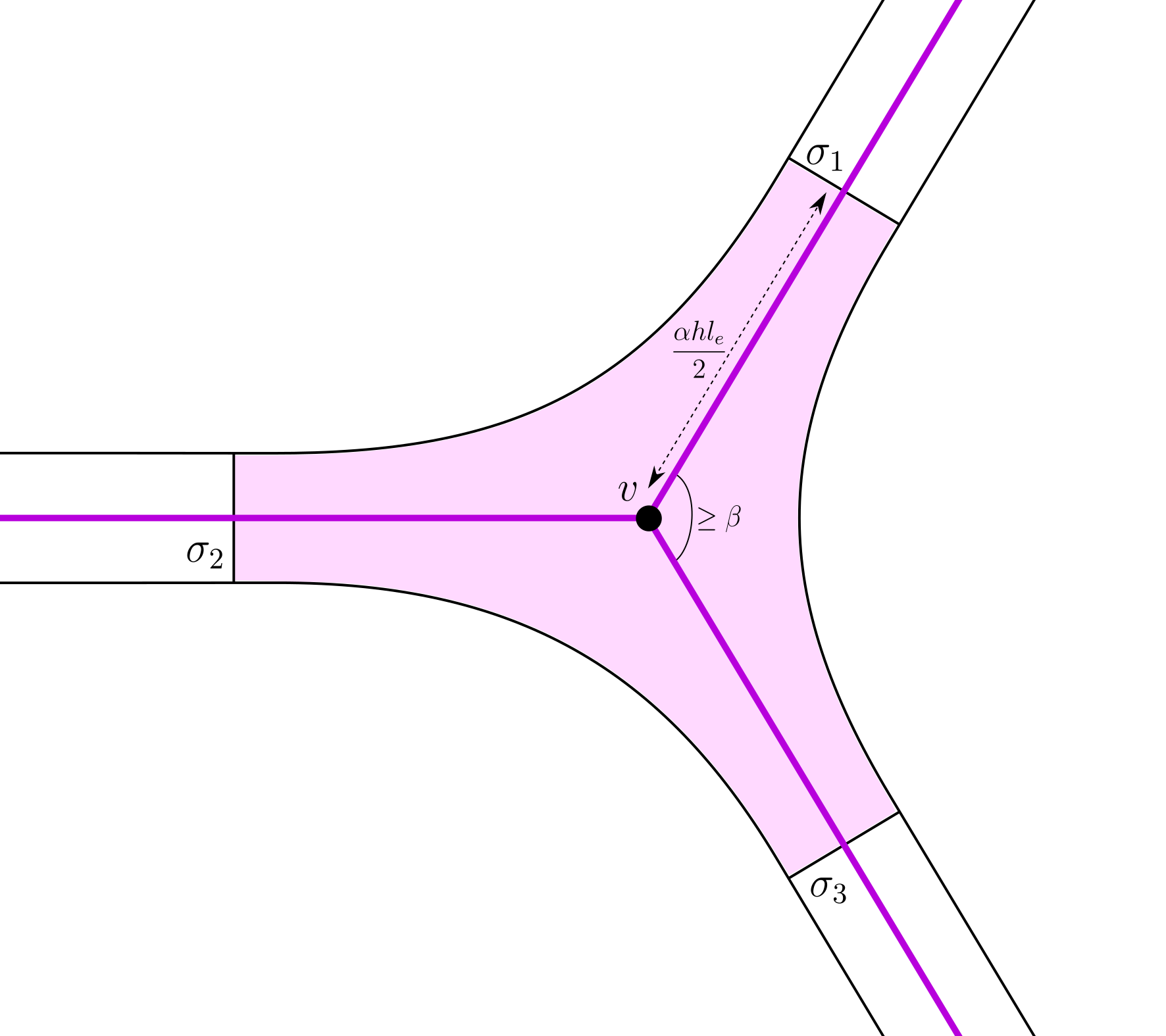}
\caption{Vertex neighborhood constructed in Step 1 of the proof of Proposition~\ref{prop:graph-like domains}. The graph $G_t$ is in dark pink while the neighborhood $U_{v,h}$ is in light pink.}
\label{fig:Uv}
\end{figure}

Let $\beta>0$ be the minimum of the angles between two adjacent edges and $\delta>0$ the minimal length of the edges in $G_t$. One checks that as long as $\tan(\beta/2)>1/(\alpha \delta)$, which can always be achieved for some large enough $\alpha$, the segments $\sigma_1,...,\sigma_d$ never intersect in the neighborhood of $v$ for all $0<h<h_0$. In particular the edge neighborhoods constructed previously never intersect. Furthermore, the pre-image of $\sigma_i$ by the contraction map $F_{v,h}(x):=hx+v$ is a fixed segment $\Sigma_i$ independent of $h$ for $i=1,...,d$. And the pre-image of the edge neighborhood $U_{e_i,h}$ by $F_{v,h}$ coincides with a rectangle near $\Sigma_i$. Besides, note that the scaling map $F_{v,h}$ satisfies the requirements of Definition~\ref{def:convergence vers graph}.

To construct $U_v$ we now just connect the pulled back segments $\Sigma_i$ for $i=1,...,d$\/ to form a $C^\infty$ Jordan curve except at the endpoints of $\Sigma_i$, where the curve locally coincides with a corner of internal angle $\pi/2$. The open set $U_v$ is the interior of the Jordan curve obtained in this way and we define $U_{v,h}$ as the image of $U_v$ by $F_{v,h}$. Note that by construction the junction between $U_v$ and $F_{v,h}^{-1}(U_{e_i,h})$ is a straight line, hence the set 
$$
\overline{\Theta}_{h,t}=\bigcup_{v\in V} \overline{U}_{v,h}\cup\bigcup_{e\in E}\overline{U}_{e,h},
$$
is $C^\infty$ regular, as desired. The resulting domain $\Theta_{h,t}$ is the interior of this set. Of course, since $G_t$ is symmetric with respect to the $x$-axis we shall assume without loss of generality that $\Theta_{h,t}$ is also symmetric. Furthermore, it is doubly-connected since it retracts by deformation on $G_t$ which is itself doubly-connected.

\medskip\noindent
\emph{Step 2 (Extension of the family of domains locally around a fixed $t_0\in T$)}: In the previous step the translation parameter $t$ was fixed. Since the second and the third points of the proposition require some regularity as $t$ varies, it is not enough to perform the previous construction for each $t\in T$ independently. Instead, we will use the domains $\Theta_{h,t_0}$ generated above for some fixed $t_0\in T$ and perform local translations to generate $\Theta_{h,t}$ for all $t$ close to $t_0$. Inspired by the proof of point (7) in \cite[Lemma~4.1]{freitas-leylekian}, we consider the transformation $\Psi_t:=\text{id}+(t-t_0)\phi_{t_0}\vec{e_x}$ where $\phi_{t_0}$ is (a smooth function symmetric with respect to the $x$-axis) supported in an annulus of width $2h_0$ around the loop of $G_{t_0}$, and equal to one in an annulus of width $h_0$ around that loop. This transformation has the effect of translating horizontally the loop of $G_{t_0}$ and the points nearby, but leaving invariant the rest of the plane. Actually, when $t$ is close enough to $t_0$, $\Psi_t$ is a $C^\infty$-diffeomorphism (by the global inverse function Theorem) depending smoothly on $t$, preserving the symmetry with respect to the $x$-axis, and sending $G_{t_0}$ to $G_t$ (the last point requires that $h_0$ is small enough, independently of $t_0,t\in T$, the set $T$ being compact, so that the rays of $G_{t_0}$ are left invariant through $\Psi_t$). We even have that the edges and the vertices of $G_t$ are the images of that of $G_{t_0}$ through $\Psi_t$. Therefore we construct $\Theta_{h,t}$ as the image of $\Theta_{h,t_0}$ through $\Psi_t$ for $t$ close to $t_0$ and for all $0<h<h_0$.

With this construction, we shall prove the third point of the proposition in a straightforward way. Indeed, it is enough to consider a $C^\infty$-diffeomorphism $\Phi_h$ mapping $\A$ to $\Theta_{h,t_0}$, preserving the symmetry with respect to the $x$-axis, and then to define $\Phi_{h,t}:=\Psi_t\circ\Phi_h$. Observe that the dependence of $\Phi_{h,t}$ on $t$ is regular since $\Psi_t$ depends linearly on $t$. The regularity of the inverse follows from the implicit function Theorem, applied to $(t,\Phi)\mapsto \Phi\circ\Phi_{h,t}$.

For the second point of the proposition, we need to argue that for each $t$ close to $t_0$ and for any sequence $(t_n,h_n)\to(t,0)$ the domains $\Theta_{h_n,t_n}$ converge to $G_t$. Naturally, we will use the sets $U_{n,v_n}:=\Psi_{t_n}(U_{v_0,h_n})$ and $U_{n,e_n}:=\Psi_{t_n}(U_{v_0,h_n})$ as the blocks in the decomposition~\eqref{eq:decomposition} of $\Theta_{h_n,t_n}$, where $v_n=\Psi_{t_n}(v_0)$ and $e_n=\Psi_{t_n}(e_0)$. To do so, we need to construct the maps $F_{n,v}$ sending a fixed reference open set $U_v$ to $U_{n,v_n}$ and the maps $F_{n,e}$ sending $(-l_e/2,l_e/2)\times(0,1)$ to $U_{n,e}$, for some vertex $v$ and some edge $e$ of $G_t$. Concerning $F_{n,v}$ there is no issue and we can simply take $v:=\Psi_t(v_0)$ and $F_{n,v}:=\Psi_{t_n}\circ F_{v_0,h_n}$, where $F_{v_0,h_n}=h_n\text{id}+v_0$ is the diffeomorphism that was defined to prove the convergence of $\Theta_{h,t_0}$ to $G_{t_0}$ in the first step of the proof (note that $v_0$ is the vertex of $G_{t_0}$ given by $v_0=\Psi_t^{-1}(v)=\Psi_{t_n}^{-1}(v_n)$). The function $F_{n,v}$ maps the reference set $U_{v_0}$ to $U_{n,v}$ and satisfies the first point of Definition~\ref{def:convergence vers graph} since

$$
(DF_{n,v})^TDF_{n,v}=h_n^2(D\Psi_{t_n})^T(D\Psi_{t_n}).
$$

For the edges, we observe that most of the sets $U_{n,e_n}$ coincide or are just translations of the set $U_{e_0,h_n}$, where $e_0=\Psi_{t_n}^{-1}(e_n)$. This actually holds for all the edges but $A_{t_n}$ and $B_{t_n}$, namely the two edges forming the bridge of $G_{t_n}$. For the translated edges, it is enough to compose $F_{e_0,h_n}$ with a translation to define $F_{n,e}$.

Thus it remains only to address the case where the edge $e$ of $G_t$ is $A_{t}$ or $B_{t}$. To simplify the discussion, we assume without loss of generality that in some neighborhood (of width $h_0$) of the $x$-axis, the function $\phi_{t_0}$ used in the definition of $\Psi_{t_0}$ does not depend on the $y$-variable. Furthermore, we recall that the functions $\psi_1$ and $\psi_2$ involved in the definition of $F_{h,e_0}$ in the first step of the proof satisfy $\psi_1'\equiv \pm 1$ and $\psi_2'\equiv 0$, since here $e_0$ is just a straight horizontal segment, namely $\Psi_t^{-1}(A_t)$ or $\Psi_t^{-1}(B_t)$. Thanks to the previous properties we realise that the pull-back of the Euclidean metric through the transformation $\Psi_{t_n}\circ F_{e_0,h_n}$ has a simple form, given by the matrix
$$
\begin{pmatrix}
(1-\alpha h_n)^2(1+(t_n-t_0)\phi_{t_0}')^2 & 0 \\
0 & h_n^2
\end{pmatrix},
$$
where $\phi_{t_0}'$ stands for $\partial_x\phi_{t_0}$. Observe that $(1+(t_n-t_0)\phi_{t_0}')$ is nothing else but the derivative of $\Psi_{t_n}|_\R^\R\circ\psi_1:[-l_{e_0}/2,l_{e_0}/2]\to \overline{e}_n\subseteq\R$. Therefore we consider the function $\varphi_t:=\psi_{1,t}^{-1}\circ \Psi_t|_\R^\R\circ\psi_1$, where $\psi_{1,t}$ corresponds to the first coordinate in the arc-length parametrisation of $\overline{e}$. This is a smooth diffeomorphism from $[-l_{e_0}/2,l_{e_0}/2]$ to $[-l_{e}/2,l_{e}/2]$ whose inverse has the derivative ${\varphi_t^{-1\prime}}=(1+(t-t_0)\phi_{t_0}')^{-1}$. Therefore the transformation defined by $F_{n,e}(x,y):=\Psi_{t_n}\circ F_{h_n,e_0}(\varphi_t^{-1}(x),y)$ for $(x,y)\in(-l_e/2,l_e/2)\times(0,1)$ is a smooth diffeomorphism such that $(DF_{n,e})^TDF_{n,e}$ has the following form:
$$
\begin{pmatrix}
(1-\alpha h_n)^2\frac{(1+(t_n-t_0)\phi_{t_0}')^2}{(1+(t-t_0)\phi_{t_0}')^2} & 0 \\
0 & h_n^2
\end{pmatrix}.
$$
Since $t_n\to t$ this proves the convergence demanded in the second point of Definition~\ref{def:convergence vers graph}. Furthermore, the expansion holds in $C^\infty(U_e)$, by smoothness of $\phi_{t_0}$. To sum up, in this step we constructed for all $t$ in a neighborhood of $t_0$ and for all $0<h<h_0$ the domains $\Theta_{h,t}$ with the three properties of the proposition, namely the double-connectedness and the symmetry with respect to the $x$-axis; the $C^\infty$ convergence of $\Theta_{h_n,t_n}$ to $G_t$ as $(t_n,h_n)\to (t,0)$; and the smoothness with a $C^1$ dependence on $t$ for each fixed value of $h$. We point out that the size of the neighborhood around $t_0$ is actually independent of $t_0$. It corresponds to the size of the neighborhood in which $\Psi_t$ remains a diffeomorphism. Similarly the smallness of $h_0$ does not depend on the choice of $t_0$.

\medskip\noindent
\emph{Step 3 (Construction of a global family of domains for all $t\in T$):} To conclude we need to turn the local construction of the second step into a global construction. To do so, we begin with $t_0\in T$. By the first and the second step we know that there exists $\epsilon$ and $h_0$ such that we shall construct the family $\Theta_{h,t}$ for all $0<h<h_0$ and all $t\in[t_0-\epsilon, t_0+\epsilon]$. We now recommence from $t_0+\epsilon$. Namely, thanks to the second step we extend the family $\Theta_{h,t_0+\epsilon}$ into a family $\Theta_{h,t}$ for all $t\in[t_0+\epsilon,t_0+2\epsilon]$ and all $0<h<h_0$. Note that we used that the neighborhood in which the extension can be performed is of the same size for $t_0+\epsilon$ and for $t_0$. Then we recommence from $t_0+2\epsilon$ and so on. Since $T$ is compact, the process ends up in a finite number of iterations. The domains $\Theta_{h,t}$ constructed in this way satisfy the three requirements of the proposition for all $t\in T$ except possibly for $t\in t_0+\epsilon\Z$. Let $\tau_k=t_0+k\epsilon$ be one of those exceptional points. Since $\tau_k\in [\tau_{k-1},\tau_k]$ the first requirement holds. The second also holds since for $(t_n,h_n)\to (\tau_k,0)$, we may split the sequence into the subsequences $(t'_n,h'_n)$ and $(t''_n,h''_n)$ where $t'_n\in[\tau_{k-1},\tau_k]$ while $t''_n\in[\tau_k,\tau_{k+1}]$. But $\Theta_{h'_n,t'_n}$ and $\Theta_{h''_n,t''_n}$ both converge to $G_{\tau_k}$ in the sense of Definition~\ref{def:convergence vers graph}, thus we conclude that also $\Theta_{t_n,h_n}$ converges to $G_{\tau_k}$. For the last requirement, we know that we have two families of transformations $\Phi_{h,t}^{(k)}$ and $\Phi_{h,t}^{(k+1)}$ mapping $\A$ to $\Theta_{h,t}$ for $t\in [\tau_{k-1},\tau_k]$ and $t\in [\tau_k,\tau_{k+1}]$, respectively. To obtain a single family $\Phi_{h,t}$ of transformations on the union of the intervals we define $\Phi_{h,t}:=\Phi_{h,t}^{(k)}$ for $t\in[\tau_{k-1},\tau_k]$ and $\Phi_{h,t}:=\Phi_{h,t}^{(k+1)}\circ(\Phi_{h,\tau_k}^{(k+1)})^{-1}\circ \Phi_{h,\tau_k}^{(k)}$ for $t\in [\tau_k,\tau_{k+1}]$. The resulting diffeomorphism depends continuously on $t$ on the reunion $[\tau_{k-1},\tau_{k+1}]$ and it is still $C^1$ in $t$ on $[\tau_{k-1},\tau_k]$ and on $[\tau_k,\tau_{k+1}]$. This concludes the proof.

\end{proof}

Thanks to the family of graph-like domains of Proposition~\ref{prop:graph-like domains} and to the results of appendix~\ref{annexe:convergence}, we are equipped to construct a Neumann \fdsh, and hence to prove Proposition~\ref{prop:FDSH}.

\begin{proof}[Proof of Proposition~\ref{prop:FDSH}]
Let us consider the planar \fgsl\/ $(G_t)_{t\in (-1,1)}$ of Lemma~\ref{lemme:aplanissement} with rays of length $R>R_0$ where $R_0$ is given by Proposition~\ref{prop:construction du graph abstrait}. In this way the nodal set of the second Neumann eigenfunction over $G_t$ lies in $I_\epsilon:=((-\epsilon,0),(\epsilon,0))$, which is a proper subset of the bridge $[W,E]$ of $G_t$, for all $t\in(-1,1)$. Without loss of generality, we assume the \fgsl\/ to be smooth and locally straight. Then, by Proposition~\ref{prop:graph-like domains} there exists a family $(\Theta_{h,t})_{t\in(-1,1),0<h<h_0}$ of smooth doubly-connected and symmetric graph-like domains, converging to the \fgsl\/ as $h\to0$, and depending on $t$ in a continuous and piecewise $C^1$ way. We claim that for any compact subset $T$ of $(-1,1)$, as long as $h_0$ is small enough, the second eigenvalue of $\Theta_{h,t}$ is simple for all $0<h<h_0$ and all $t\in T$. This can be seen by contradiction since otherwise we would have a sequence $(t_n,h_n)\to(t,0)$, for some $t\in T$, such that $\Theta_{t_n,h_n}$ has a multiple second eigenvalue. But in view of Theorem~\ref{thm:convergence} in appendix, this would mean that $G_t$ has also a multiple first eigenvalue, which would contradict Proposition~\ref{prop:construction du graph abstrait}, due to the choice of $R>R_0$. In what follows we choose $T$ as a compact interval containing $(-\epsilon,\epsilon)$ strictly, in such a way that the junction point $J_t$ of $G_t$ remains at a posivite distance of $I_\epsilon$ when $t$ is close to $\inf T$ and of $\sup T$. We then define the family of domains $\Omega_t:=\Theta_{h_0/2,t}$, for all $t\in T$.

By construction the domains $\Omega_t$ are smooth, doubly-connected, symmetric with respect to the $x$-axis, and have a simple second eigenvalue. In other words they satisfy the first and the second points in the definition of a Neumann \fdsh. To prove the third point we use the transversal convergence of the eigenfunctions over $\Theta_{h,t}$ as $h$ goes to zero, provided by Proposition~\ref{prop:convergence}. This will show that (up to shrinking $h_0$) the second eigenfunction over $\Omega_t$ must change sign along the open line segment $(W,E)$. Indeed, assume by contradiction that there were a sequence $(t_n,h_n)\to(t,0)$ for some $t\in T$ such that the second eigenfunction $u_n$ over $\Theta_{h_n,t_n}$, normalised in $L^2(\Theta_{h_n,t_n})$, does not change sign on $(W,E)$. Let $v_n^A$ and $v_n^B$ denote the pull-back of $u_n$ by $F_{A,h_n}$ and $F_{B,h_n}$ respectively, where $A$ and $B$ are the horizontal edges (of length $a$ and $b$) forming the bridge of $G_t$ and where $F_{A,h_n}$ and $F_{B,h_n}$ are the transformations involved in the convergence of $\Theta_{h_n,t_n}$ to $G_t$ (see Definition~\ref{def:convergence vers graph}). By hypothesis for all $x_0\in(-a/2,a/2)$ and all $x_1\in(-b/2,b/2)$, there exist $y_0^n,y_1^n\in (0,1)$ such that $v_n^A(x_0,y_0^n)$ and $v_n^B(x_1,y_1^n)$ are of same sign. Yet for a.e. $x_0$ and $x_1$, Proposition~\ref{prop:convergence} shows that $\sqrt{h_n}v_n^A(x_0,\cdot)\to u|_A(x_0)$ and $\sqrt{h_n}v_n^B(x_1,\cdot)\to u|_B(x_1)$, uniformly in $[0,1]$. Here $u$ is the second $L^2$-normalised eigenfunction over $G_t$, and its restrictions to $A$ and $B$ are considered as functions over $(-a/2,a/2)$ and $(-b/2,b/2)$ respectively, following Remark~\ref{remarque:convergence vers graphe}. As a consequence $u|_A(x_0)$ and $u|_B(x_1)$ are also of same sign. This contradicts Proposition~\ref{prop:construction du graph abstrait} and the choice of $R>R_0$.

As a result, the first eigenfunction on $\Omega_t$ changes sign along $(W,E)$ for all $t\in T$, up to shrinking $h_0$ and the third point in Definition~\ref{def:sliding domains} holds true. The fourth point follows from the fact that the diffeomorphisms $\Phi_t:=\Phi_{h_0/2,t}$ provided by Proposition~\ref{prop:graph-like domains}, together with their inverse, depend continuously and piecewise $C^1$ over $t$. Indeed, after an adaptation of \cite[Lemma~A.1]{freitas-leylekian} to the Neumann case -- see Lemma~\ref{lemme:continuité fonction propres par rapport à t} below -- one obtains that
$$t\in T\mapsto u_t\circ\Phi_t\in H^2(\A)\subseteq C(\overline{\A})$$
is continuous (and piecwise $C^1$), for some family of eigenfunctions $u_t$ over $\Omega_t$. Now it only remains to check the fifth and last point in Definition~\ref{def:sliding domains}. Let $t_-\in T$ be such that the junction point of $G_{t_-}$ lies at the left of $I_\epsilon$, namely the segment where the nodal set of the eigenfunction of $G_{t_-}$ is located. Such a $t_-$ exists in view of our demands on the compact set $T$. Similarly $t_+\in T$ is taken such that the junction point of $G_{t_+}$ is at the right of $I_\epsilon$. Observe that in this setting we have $I_\epsilon\subseteq B$ for $t=t_-$ and $I_\epsilon \subseteq A$ for $t=t_+$. Thus the identification $B\cong[-b/2,b/2]$ naturally maps the segment $I_\epsilon$ to some interval $(\beta_1,\beta_2)\subseteq(-b/2,b/2)$ for $t=t_-$, while the identification $A\cong[-a/2,a/2]$ maps this segment to $(\alpha_1,\alpha_2)\subseteq(-a/2,a/2)$ for $t=t_+$. We shall now prove that the nodal set over $\Omega_{t_\pm}$ is also located around $I_\epsilon$ (up to shrinking $h_0$). Indeed, let $h_n\to 0$, consider the normalised eigenfunction over $\Theta_{h_n,t_+}$ and pull it back through the transformation $F_{A,h_n}$ to define $v_n^A$. Once again, by the strong transversal convergence of Proposition~\ref{prop:convergence}, $\sqrt{h_n}v_n^A(x,\cdot)$ converges uniformly to $u|_A(x)$, for a.e. $x\in(-a/2,a/2)$. Take first $x=x_1$ for some point $x_1$ close to $\alpha_1$ and where the previous convergence holds. When $n$ is large enough the sign of $v_n^A(x_1,\cdot)$ is dictated by the sign of $u|_A(x_1)$. Then, take $x=x_2$ with $x_2$ close to $\alpha_2$, so that the sign of $v_n^A(x_2,\cdot)$ is the same as that of $u|_A(x_2)$. Since $u|_A(x_1)$ and $u|_A(x_2)$ have opposite sign, $v_n^A$ also have opposite sign on the section $\{x_1\}\times(0,1)$ and on the section $\{x_2\}\times(0,1)$. Hence its nodal line must be contained in $(\alpha_1,\alpha_2)\times(0,1)$. This shows in turn that the nodal line of the eigenfunction over $\Theta_{h_n,t_-}$ touches only the left boundary component of $\Theta_{h_n,t_-}$, that is the connected component of the boundary  which meets $U_{h_n,A}$. Analogously the nodal line over $\Theta_{h_n,t_+}$ touches exclusively the right boundary component, i.e. the connected component of the boundary meeting $U_{h_n,B}$. We recall that $U_{h_n,A}$ and $U_{h_n,B}$ are the blocks associated to the edges $A$ and $B$, respectively, in the convergence of $\Theta_{h_n,t_\pm}$ to $G_{t_\pm}$. Up to shrinking $h_0$, this shows that the nodal line over $\Omega_{t_-}$ meets the right boundary component whereas it touches the left one over $\Omega_{t_+}$. Since (up to an inversion) $\Phi_{t_\pm}^{-1}$ maps those boundary components respectively to the inner and to the outer boundary components of $\A$ for all $t\in T$, this proves the last requirement in Definition~\ref{def:sliding domains}. In other words $(\Omega_t)_{t\in T}$ constitutes a Neumann \fdsh.
\end{proof}

\begin{lemme}\label{lemme:continuité fonction propres par rapport à t}
Let $D$ be a $C^k-$regular bounded open set ($k\geq2$) and $\Phi_t:\R^d\to\R^d$ be a family of $C^k$-diffeomorphisms, for $t$ in an open interval $T$, such
that $(T\ni t\mapsto\Phi_t,\Phi_t^{-1}\in C^k(\R^d,\R^d)^2)$ is $C^1$. Define $\Omega_t:=\Phi_t(D)$ and assume that for $t_0\in T$, $\lambda$ is a simple eigenvalue
on $\Omega_{t_0}$. Then, for $t$ in some neighborhood of $t_0$, there exists a family of eigenvalues $\lambda_t$ on $\Omega_t$ and a family
of associated eigenfunctions $u_t$ such that $\lambda_{t_0}=\lambda$ and such that the following function is $C^1$ in the neighborhood of $t_0$:
$$
t\mapsto (\lambda_t,u_t\circ\Phi_t)\in \R\times H^k(D)
$$
\end{lemme}

\begin{proof}
Inspired by \cite[p.203--204]{henrot-pierre}, we define the functional $\mathcal{F}:T\times H^k(D)\times\R\to \mathcal{B}(D)\times\R$ given by
$$
\mathcal{F}(t,v,\mu)=\left(\mathcal{L}_tv-\mu vJ(t); \int_Dv^2J(t)\right).
$$
Here $\mathcal{B}(D):=H^{k-2}(D)\oplus H^{k-3/2}(\partial D)$ is the set of linear forms $\phi$ for which there exists a (unique) pair $f\in H^{k-2}(D)$ and $g\in H^{k-3/2}(\partial D)$ such that
$$
\langle \phi,\varphi\rangle=\int_Df\varphi+\int_{\partial D}g\varphi,\quad\forall\varphi\in H^1(D).$$
Furthermore $J(t)$ is the Jacobian of $\Phi(t)$ and $\mathcal{L}_tv$ is the element of $\mathcal{B}(D)$ given by
$$
\langle\mathcal{L}_tv,\varphi\rangle:=\int_DA(t)\nabla v\cdot\nabla \varphi =-\int_D\dv(A(t)\nabla v)\varphi+\int_{\partial D}A(t)\nabla v\cdot n\varphi,\quad\forall\varphi\in H^1(D),
$$
where $A(t):=J(t)((D\Phi_t^{-1})^TD\Phi_t^{-1})\circ\Phi_t$ while $n$ denotes the outward normal unit vector. Note that $u$ is an $L^2$-normalised eigenfunction associated with $\lambda$ over $\Omega_t$ if and only if $\mathcal{F}(t,u\circ\Phi_t,\lambda)=(0,1)$. The functional $\mathcal{F}$ is well defined since $J(t)$ belongs to $W^{k-2,\infty}_{loc}$ and $A(t)$ to $W^{k-1,\infty}_{loc}$. Furthermore, it is $C^1$ since the dependence of $\Phi_t$ and $\Phi_t^{-1}$, together with their derivatives, is $C^1$ with respect to~$t$. The derivative at $(t_0,v_0,\mu_0)$ with respect to $(v,\mu)$ is
$$
D\mathcal{F}_{t_0,v_0,\mu_0}(v,\mu)=\left(\mathcal{L}_{t_0}v-\mu_0vJ(t_0)-\lambda  v_0J(t_0); 2\int_Dvv_0J(t_0)\right)
$$
We need to argue that when $\mu_0=\lambda$ and $v_0=u_0\circ\Phi_{t_0}$, $u_0$ being an $L^2$-normalised eigenfunction associated with $\lambda$ on $\Omega_{t_0}$, the derivative defines an isomorphism from $H^k(D)\times\R$ to $\mathcal{B}(D)\times\R$. This amounts to prove that for any $f\in H^{k-2}(D)$, $g\in H^{k-3/2}(\partial D)$, and any $c\in\R$, there exists a unique pair $\mu\in\R$ and $v\in H^{k}(D)$ such that
$$
\left\{
\begin{array}{rcll}
-\dv(A(t_0)\nabla v) & = & f+\lambda vJ(t_0)+\mu  v_0 J(t_0)& \text{in }D,\\
A(t_0)\nabla v\cdot n & = & g & \text{on }\partial D,
\end{array}
\right.
$$
together with the constraint $\int_Dvv_0J(t_0)=c$. Note that up to replacing $v$ with $V=v-G$ and $f$ with $F=f+(\lambda+\epsilon) G$, where $G\in H^k(D)$ is the $\epsilon$-harmonic extension of $g$ (namely the unique function such that $(-\dv(A(t_0)\nabla v)+\epsilon)G=0$ and $A(t_0)\nabla G\cdot n=g$) it is enough to consider the case where $g=0$. Then, since $\lambda$ is a simple eigenvalue associated with $v_0$, Fredholm's alternative (applied to the resolvent of the operator $-\dv A(t_0)+c$ with Neumann boundary condition, for some $c>0$, which is a compact operator on $L^2(D)$) indicates that there exists a solution $V\in L^2(D)$ if and only if the following orthogonality condition holds:
$$
\int_D(F+\mu v_0J(t_0))v_0=0 \quad\Leftrightarrow\quad \mu+\int_Dfv_0+\int_{\partial D} gv_0=0. 
$$
This determines $\mu$. Then the solutions are of the form $V=U+\alpha v_0$, for some uniquely determined $U\in L^2(D)$, and for any $\alpha\in\R$. But the constraint $\int_Dvv_0J(t_0)$ imposes a unique value for~$\alpha$. In turn, there exists a unique solution $v=V+G$ to the initial equation, and by elliptic regularity for the Neumann problem \cite[Theorem~15.2]{agmon-douglis-nirenberg}, this $v$ belongs to $H^k(D)$. This proves that $D\mathcal{F}_{t_0,v_0,\mu_0}$ is an isomorphism, and by the implicit function Theorem there exists a unique $C^1$ path $t\mapsto (v_t,\lambda_t)\in H^k(D)\times\R$, for $t$ close to $t_0$, such that $\mathcal{F}(t,v_t,\lambda_t)=(0,1)$ with $\lambda_{t_0}=\lambda$. By definition of $\mathcal{F}$, this establishes the result.
\end{proof}



\begin{appendix}
\section{Convergence of graph-like domains}\label{annexe:convergence}

In this appendix we recall and extend results contained in the paper of Post \cite{post}. For the record, let us also mention the monograph~\cite{post-livre}, on the same topic and by the same author. We will first define the notion, for a family of domains, of convergence towards a planar graph. The purpose of this definition is to ensure that the Neumann spectrum of the domains converge to that of the graph and that the eigenfunctions also converge to the eigenfunctions of the graph, in some sense. Since the limit of the domains is lower dimensional the correct notion of convergence of eigenfunctions is not trivial. In \cite{post} some $L^2$ convergence is proved, see Theorem~\ref{thm:convergence}. In the present paper however, the $L^2$ convergence is not enough and we needed to prove what we call a strong \emph{transversal} convergence, see Proposition~\ref{prop:convergence}. Let us first give the definition of convergence towards a graph. The definition is basically made in order that all the assumptions of \cite{post}, in particular (G1)--(G7), are fulfilled.

\begin{definition}\label{def:convergence vers graph}
Let $G=(V,E)$ be a planar connected metric graph and let $(\Theta_h)_{0<h<h_0}$ be a family of planar connected bounded open sets, for some $h_0>0$. We say that $\Theta_h$ converges to $G$ as $h\to0$ whenever for each $0<h<h_0$ there exists a decomposition
\begin{equation}\label{eq:decomposition}
\overline{\Theta}_h=\bigcup_{v\in V} \overline{U}_{v,h}\cup\bigcup_{e\in E}\overline{U}_{e,h},
\end{equation}

where $U_{v,h}$ and $U_{e,h}$ are open sets which are all mutually disjoint. Their closures are also assumed disjoint, except when $v$ is an endpoint of $e$, in which case $\partial U_{v,h}\cap\partial U_{e,h}$ is diffeomorphic to $(0,1)$. Furthermore, $U_{v,h}$ and $U_{e,h}$ are respectively the image of a fixed open set $U_v$ independent of $h$ and of $U_e:=(-l_e/2,l_e/2)\times(0,1)$ (where $l_e$ is the length of $e$) through $C^1$-differomorphisms $F_{v,h}$ and $F_{e,h}$ such that
\begin{enumerate}
\item the inequality $c_- h^2\leq (D_xF_{h,v})^TD_x F_{h,v}\leq c_+ h^2$ holds
in the sense of quadratic forms for some $c_\pm>0$ independent of $v\in V$ and $x\in U_v$.
\item the following expansion holds uniformly with respect to $x\in U_e$ and to $e\in E$:
$$
(D_xF_{h,e})^TD_x F_{h,e}\underset{h\to 0}{=}
\begin{pmatrix}
1+\so(1) & \so(h) \\
\so(h) & h^2+\so(h^2)
\end{pmatrix}.
$$
\end{enumerate}
Lastly, the convergence is said to hold in $C^{k+1}$ if $F_{e,h}$ is a $C^{k+1}$-diffeomorphism and if the previous asymptotic expansion holds in $C^k(U_e)$.
\end{definition}

As annouced, thanks to \cite{post} this definition allows to prove some key result on the convergence of Neumann eigenvalues and eigenfunctions for a family of domains converging to a planar graph.

\begin{theoreme}[\cite{post}]\label{thm:convergence}
Let $(\Theta_h)_{0<h<h_0}$ be a family of planar domains converging to a planar graph $G$. Then $\mu_n(\Theta_h)$ converges to the $n$-th eigenvalue of\/ $G$ as $h$ goes to zero, and if that eigenvalue is simple there exists a family of $L^2(\Theta_h)$-normalised eigenfunctions $u_h$ associated with $\mu_n(\Theta_h)$ such that
$$
\|u_h-\overline{u}^h/\sqrt{h}\|_{L^2(\Theta_h)}\to0,
$$
where $\overline{u}^h$ is the constant transversal extension of the $L^2(G)$-normalised $k$-th eigenfunction $u$ on $G$, defined for $x\in\Theta_h$ by
$$
\overline{u}^h(x)=
\begin{cases}
u|_e(y) & \text{if } x=F_{h,e}(y,z)\in U_{e,h}\text{ for some }e\in E, \\
0 & \text{otherwise}.
\end{cases}
$$
\end{theoreme}

\begin{remarque}\label{remarque:convergence vers graphe}
\begin{enumerate}
\item Note that the function $u|_e$ is seen as a function defined over $(-l_e/2,l_e/2)$, thanks to the identification $e\cong[-l_e/2,l_e/2]$.
\item Observe that the functions $u_h$ and $\overline{u}^h/\sqrt{h}$ pointwise diverge and yet remain $L^2$-normalised as $h\to0$. If instead they remained bounded, the convergence would hold trivially since the measure of $\Theta_h$ goes to zero.
\end{enumerate}
\end{remarque}

\begin{proof}
The result follows from \cite[Theorem~2.13]{post}, as long as assumptions (G1)--(G7) of \cite{post} are satisfied. Note that the manifold $F$ of \cite[section~2.2]{post}, of total measure one, corresponds here to $(0,1)$ and hence that its dimension is $m=1$. Assumptions (G1)--(G3) are a consequence of the compactness of the metric graph $G$. Assumption (G4) comes from the asymptotic expansion of $(D_xF_{h,e})^TD_x F_{h,e}$, which is the matrix associated with the metric on $U_e$ (note that in our setting $r_e\equiv1$). Assumptions (G5) and (G6) follow from the inequality $c_-h^2\leq (D_xF_{h,v})^TD_x F_{h,v}\leq c_+h^2$ and from the fact that $(D_xF_{h,v})^TD_x F_{h,v}$ is the matrix of the metric on $U_v$. Assumption (G7) comes, once again, from the compactness of the graph. Therefore we conclude that the $n$-th eigenvalue on $\Theta_h$ converge to the $n$-th eigenvalue on $G$ and that by \cite[Theorem~A.12]{post} and \cite[equation~(2.7)]{post} there exists for each $h$ an $n$-th eigenfunction $v_h$ associated with $\mu_n(\Theta_h)$ such that
$$
\|v_h-\overline{u}^h/\sqrt{h}\|_{L^2(\Theta_h)}\to 0.
$$
From this convergence and from the asymptotic expansion of Definition~\ref{def:convergence vers graph}, a straightforward computation shows that
$$
\|v_h\|^2_{L^2(\Theta_h)}\sim\|\overline{u}^h/\sqrt{h}\|^2_{L^2(\Theta_h)}\sim h^{-1}\sum_{e\in E}\int_{-l_e/2}^{l_e/2}\int_{0}^1 u(y)^2\sqrt{h^2+\so(h^2)}\textup{d}y\textup{d}z\sim 1,
$$
where we used that $u$ is normalised in $L^2(G)$. Taking $u_h:=v_h/\|v_h\|_{L^2(\Theta_h)}=v_h+\so(1)$ as the $L^2(\Theta_h)$-normalised eigenfunction we thus obtain the desired result.
\end{proof}

The purpose of the next proposition is to improve the $L^2$ type convergence of Theorem~\ref{thm:convergence} and to prove what is called a strong transversal convergence.

\begin{proposition}\label{prop:convergence}
Let $(\Theta_h)_{0<h<h_0}$ be a family of domains converging in $C^{k+2}$ to a planar graph $G$, for some $k\in\N$. Then, with the assumptions and notations of Theorem~\ref{thm:convergence}, $\sqrt{h}u_h$ converges transversally almost everywhere to $u$ in $C^{k+1,\alpha}$. This means that for any edge $e$ of\/ $G$, if $v_h=u_h\circ F_{e,h}$ denotes the pull-back of the eigenfunction $u_h$ in $U_e$, then for almost every $x_0\in(-l_e/2,l_e/2)$ we have that $\sqrt{h}v_h(x_0,\cdot)\to u(x_0)$ in $C^{k+1,\alpha}([0,1])$.
\end{proposition}

The proof follows from elliptic regularity results applied to the slab obtained by stretching horizontally the rectangle $U_e=(-l_e/2,l_e/2)\times(0,1)$. The end of the proof is based on results applied to the maximal function of Hardy and Littlewood. We recall that for a function $f:I\to\R$ defined on an open interval $I$, the maximal function is the function $Mf:I\to\R$ given by
\begin{equation}\label{eq:maximale}
Mf(x)=\sup_{\gamma>0}\frac{1}{2\gamma}\int_{x-\gamma}^{x+\gamma}|f(t)|\mathrm{d}t
\end{equation}
An important result with respect to the maximal function is that for any $f\in L^1(I)$ and any $\epsilon>0$, $|\{Mf>\epsilon\}|\leq \frac{C}{\epsilon}\|f\|_{L^1(I)}$, see e.g. \cite{melas2003}. In particular, if $f_n\to 0$ in $L^1(I)$ then $Mf_n\to 0$ in measure, namely for all $\epsilon>0$, $|\{|Mf_n|\geq \epsilon\}|\to0$. Furthermore, we recall that convergence in measure implies a.e. convergence, up to a subsequence \cite[Theorem~2.2.5]{bogachev}.

\begin{proof}[Proof of Proposition~\ref{prop:convergence}]
By Theorem~\ref{thm:convergence} we already know that $\sqrt{h}v_h$ converges to $\overline{v}^h$ in $L^2(U_e)$, where $\overline{v}^h=\overline{u}^h\circ F_{e,h}$. To improve this convergence we aim at applying elliptic estimates. Indeed, the function $u_h$ satisfies the eigenvalue equation for the Laplacian over $U_{e,h}$ with eigenvalue $\mu_h$. When pulled back on $U_e$, the equation is transformed into
$$
-\Delta_h v_h=\mu_h v_h,
$$
where $\Delta_h$ is the Laplace-Beltrami operator associated with the metric induced by $F_{e,h}$, whose matrix is of the following form by virtue of the convergence of $\Theta_h$ to $G$:
$$
\begin{pmatrix}
1+\so(1) & \so(h)\\ \so(h) & h^2+\so(h^2)
\end{pmatrix}
$$
The issue is that this degenerates as $h$ goes to zero, hence elliptic estimates cannot be applied directly. To resolve this problem we stretch the plane horizontally around a given point $x_0\in (-l_e/2,l_e/2)$. More precisely, for a function $v$ defined on $U_e$ we introduce $\tilde{v}(x,y)=v(hx+x_0,y)$, defined on the horizontally blown up rectangle $\tilde{U}_{h,e}:=(-l_e/2-x_0,l_e/2-x_0)/h\times(-1,1)$. With this transformation $\tilde{v}_h$ satisfies
$$
-\tilde{\Delta}_h\tilde{v}_h=h^2\mu_h\tilde{v}_h,
$$
where $\tilde{\Delta}_h$ is now the Laplace-Beltrami operator associated with the metric
$$
\begin{pmatrix}
1+\so(1) & \so(1)\\ \so(1) & 1+\so(1)
\end{pmatrix}.
$$
Note that the above $\so(1)$ are understood in the sense of $C^{k+1}$, which means that there $C^{k+1}(\tilde{U}_{h,e})$ norm converges to zero. On the other hand, recall that $\overline{v}^h(x,y)=u(x)$ for all $(x,y)\in U_e$, and that $u$ satisfies the equation $-u''=\mu u$ where $\mu$ is the eigenvalue of the graph. This shows that the horizontal blow-up $\tilde{\overline{v}}^h$ of $\overline{v}^h$ satisfies over $\tilde{U}_{h,e}$ the equation
$$
-\Delta\tilde{\overline{v}}^h=h^2\mu\tilde{\overline{v}}^h,
$$
We now set $w_h=\sqrt{h}v_h-\overline{v}^h$ and observe that we have
$$
-\tilde{\Delta}_h\tilde{w}_h=h^2\mu_h\tilde{w}_h+h^2(\mu_h-\mu)\tilde{\overline{v}}^h+(\Delta-\tilde{\Delta}_h)\tilde{\overline{v}}^h.
$$
Recall that this elliptic equation is defined over the rectangle $\tilde{U}_{h,e}$ which is increasing as $h$ goes to zero and converges to the horizontal slab $S=\R\times(0,1)$. Thus any bounded open subset $\omega$ of the slab is eventually contained in  $\tilde{U}_{e,h}$. Furthermore the term $h^2(\mu_h-\mu)\tilde{\overline{v}}^h$ converges to zero in $C^\infty(\omega)$ since $\tilde{\overline{v}}^h$ and its derivatives are bounded by that of $u$. Similarly the term $(\Delta-\tilde{\Delta}_h)\tilde{\overline{v}}^h$ converges to zero in $C^{k}(\omega)$ since the coefficients of $\Delta$ and $\tilde{\Delta}_h$ differ only by $\so(1)$, in view of the structure of the metric. Note that we used the fact that the convergence of the domains to the graph holds in $C^{k+2}$, since the coefficients in $\tilde{\Delta}_h$ involve derivatives of the entries of the matrix. Lastly, $w_h$ goes to zero in $L^2(U_e)$ and hence $\sqrt{h}\tilde{w}_h$ converges to zero in $L^2(\omega)$.

Through local elliptic estimates for the Neumann problem, one can show that $\sqrt{h}\tilde{w}_h\to0$ in $C^{k+1,\alpha}(\omega)$. But that is not enough for our purpose since it only gives the useless result $hv_h(x_0,\cdot)\sim\sqrt{h}u(x_0)\to0$. Instead, we aim at proving that $\tilde{w}_h\to0$ in $C^{k+1,\alpha}$. To that end, we consider as the window $\omega$ the rectangle $\omega=(-1,1)\times(0,1)$ and we observe that by the Cauchy-Schwarz inequality,
$$
\(\int_\omega|\tilde{w}_h|\right)^2\leq 4\int_\omega\tilde{w}_h^2=4\int_{x_0-h}^{x_0+h}\int_{0}^1\frac{w_h^2}{h}\leq 4\sup_{\gamma>0}\int_{x_0-\gamma}^{x_0+\gamma}\int_{0}^1\frac{w_h^2}{\gamma}= 8Mf_h(x_0),
$$
where $Mf_h$ is the maximal function~\eqref{eq:maximale} associated to $f_h:x\mapsto\int_{0}^1w_h^2(x,y)\mathrm{d}y$. Recall that since $w_h\to0$ in $L^2(U_e)$ we have $f_h\to0$ in $L^1((-l_e/2,l_e/2))$ and consequently $Mf_h\to0$ in measure. Then, up to extracting a subsequence, $Mf_h$ converges a.e. to zero. This shows that as long as $x_0\in(-l_e/2,l_e/2)$ is chosen outside some set of measure zero, the function $\tilde{w}_h$ converges to zero in $L^1(\omega)$, up to a subsequence.

By a local elliptic estimate in $\omega$ (see Theorem~15.2 and the italicised remark p.~706 in~\cite{agmon-douglis-nirenberg}) we now conclude that $\tilde{w}_h$ converges to zero in $H^2(\omega)$, up to a subsequence. By Sobolev embeddings, the convergence then holds uniformly in $\omega$, and hence in $L^p(\omega)$ for any $p>2$. Iterating elliptic regularity arguments, the convergence holds in $W^{l,p}(\omega)$ for $l=0,...,k+2$. Note that the iteration is limited by the term $(\Delta-\tilde{\Delta}_h)\tilde{\overline{v}}^h$ in the second member, which converges only in $C^k(\omega)$. This shows that $\tilde{w}_h$ goes to zero in $C^{k+1,\alpha}(\overline{\omega})$ for all $0<\alpha<1$, by Sobolev embeddings. We thus conclude that for a.e. $x_0\in(-l_e/2,l_e/2)$ the convergence of $\sqrt{h}v_h(x_0,\cdot)$ to $u(x_0)$ holds in $C^{k+1,\alpha}([0,1])$, up to a subsequence.

More precisely, we showed that for each sequence $h_n\to0$ there exists a subsequence, still denoted $h_n$, and a negligible set of $(-l_e/2,l_e/2)$ such that for all $x_0$ outside this set $\sqrt{h_n}v_{h_n}(x_0,\cdot)\to u(x_0)$. By uniqueness of the limit, we actually conclude that the whole sequence $\sqrt{h}v_{h}(x_0,\cdot)$ must converge to $u(x_0)$ as $h$ goes to zero, for $x_0$ outside some negligible set.

\end{proof}

\begin{remarque}
We would like to point out that the procedure followed in the above proof, namely pulling back the problem in the fixed rectangle $U_e$ and then in the elongating rectangle $\tilde{U}_{e,h}$, can be understood in a simplified way, in the case where the edge $e$ is already a straight segment and where the map $F_{e,h}$ corresponds to a vertical stretching, of factor $1/h$. Indeed, in this case the rectangle $\tilde{U}_{e,h}$ is nothing else but a dilation of factor $1/h$ of the initial edge block $U_{e,h}$, and working inside the window $\omega$ amounts to zooming around a particular point of the edge, at scale $h$. Of course it is not possible to apply the elliptic estimates in the shrinking window, and that is why we needed to rescale it. We also recall that it was very important in the proof that not only $\sqrt{h}\tilde{w}_h$ converges to zero, but that $\tilde{w}_h$ converges to zero in $L^1(\omega)$. In this respect, the choice of a window $\omega$ that is bounded not only in the $y$, but also in the $x$ direction (which corresponds to zooming at the same scale in both directions) was a key feature of the proof. We are not able to give an intuitive justification of this last technical trick.
\end{remarque}
\end{appendix}

\section*{Acknowledgements}
This work was partially supported by the Funda\c{c}\~ao para a Ci\^encia e a Tecnologia, Portugal, via the research centre
GFM, reference UID/00208/2025.

\printbibliography



\end{document}